\documentclass[9pt]{amsart}

\usepackage[usenames,dvipsnames,svgnames,table]{xcolor}
\usepackage{amssymb}
\usepackage{mathtools}

\newcommand{\R}{{\mathbb{R}}}
\newcommand{\N}{{\mathbb{N}}}
\newcommand{\<}{\langle}
\renewcommand{\>}{\rangle}

\newcommand{\abs}[1]{\left\vert#1\right\vert}
\newcommand{\xii}{{\abs{\xi}}}

\def\<#1\>{\left\langle#1\right\rangle }
\renewcommand{\doteq}{{\,\mathrm{:=}\,}}

\newcommand{\eps}{\varepsilon}

\newcommand{\sinc}{{\mathrm{sinc}}}
\newcommand{\Str}{{0}}%{{\mathrm{Str}}}
\newcommand{\crit}{{\mathrm{crit}}}
\newcommand{\conf}{{\mathrm{conf}}}

\theoremstyle{plain}
\newtheorem{theorem}{Theorem}
\newtheorem{proposition}{Proposition}
\newtheorem{lemma}{Lemma}
\newtheorem{corollary}{Corollary}
\theoremstyle{definition}
\newtheorem{Def}{Definition}

\theoremstyle{remark}
\newtheorem{remark}{Remark}[section]

\title%[The semilinear wave with critical dissipation]{%$L^p-L^q$ estimates for t
[The semilinear Euler-Poisson-Darboux equation]{The semilinear Euler-Poisson-Darboux equation: \\ a case of wave with critical dissipation}

\author{Marcello D'Abbicco}

\begin{document}

\begin{abstract}
In this paper we study the existence of global-in-time energy solutions to the Cauchy problem for the Euler-Poisson-Darboux equation, with a power nonlinearity:
\[ u_{tt}-u_{xx} + \frac\mu{t}\,u_t = |u|^p \,, \quad t>t_0, \ x\in\R\,.\]
Here either~$t_0=0$ (singular problem) or~$t_0>0$ (regular problem). This model represents a wave equation with critical dissipation, in the sense that the possibility to have global small data solutions depend not only on the power~$p$, but also on the parameter~$\mu$. We prove that, assuming small initial data in~$L^1$ and in the energy space, global-in-time energy solutions exist for~$p>p_\crit =\max\{p_\Str(1+\mu),3\}$, for any~$\mu>0$, where~$p_\Str(k)$ is the critical exponent for the semilinear wave equation without dissipation in space dimension~$k$, conjectured by W.A. Strauss, and~$3$ is the critical exponent obtained by H. Fujita for semilinear heat equations. We also collect some global-in-time existence result of small data solutions for the multidimensional EPD equation
\[ u_{tt}-\Delta u + \frac\mu{t}\,u_t = |u|^p \,, \quad t>t_0, \ x\in\R^n\,,\]
with powers~$p$ greater than Fujita exponent and sufficiently large~$\mu$.
\end{abstract}



\maketitle

\section{Introduction}\label{sec:Main}

In this paper, we study global-in-time existence of small data solutions to the Cauchy problem for the Euler-Poisson-Darboux equation with a power nonlinearity:
\begin{equation}\label{eq:EPD}
u_{tt}-\triangle u+ \dfrac\mu{t}\, u_t= f(u);\\
\end{equation}
here~$\mu>0$ and~$f(u)=|u|^p$ or, more in general, $f$ is locally Lipschitz-continuous and
\begin{equation}\label{eq:fu}
f(0)=0, \quad |f(u)-f(w)|\leq C\,|u-w|\big(|u|^{p-1}+|w|^{p-1}\big)\,,
\end{equation}
for some~$p>1$. The importance of this semilinear model is that \emph{it represents a bridge} across the rift that lies between pure semilinear wave models ($\mu=0$) and semilinear wave models whose asymptotic profile is described by a diffusive model (sufficiently large~$\mu$). The transition from one model to the other is described by how the critical exponent changes as the dissipation parameter~$\mu$ enlarges, up to some threshold.

In this paper, we consider both the singular Cauchy problem
\begin{equation}\label{eq:CPsing}
\begin{cases}
u_{tt}-\triangle u+ \dfrac\mu{t}\, u_t= f(u), & t>0, \ x\in\R^n\,,\\
u(0,x)=u_0(x)\,, \quad u_t(0,x)=0\,,
\end{cases}
\end{equation}
and the regular Cauchy problem
\begin{equation}\label{eq:CP}
\begin{cases}
u_{tt}-\triangle u+ \dfrac\mu{t}\, u_t= f(u), & t\geq t_0>0, \ x\in\R^n\,,\\
u(t_0,x)=0\,, \quad u_t(t_0,x)=u_1(x)\,.
\end{cases}
\end{equation}
The study of the solution to the singular linear Cauchy problem, i.e., $f=0$ in~\eqref{eq:CPsing}, goes back to the first investigations of Euler~\cite{Euler}, Poisson~\cite{Poisson} and Darboux~\cite{Darboux} in space dimension~$n=1$, later extended to the multidimensional case~$n\geq2$ by A. Weinstein~\cite{Wein} and other authors, see, in particular, \cite{DW53} and the references therein. The study of the solution to the regular linear Cauchy problem, i.e., $f=0$ in~\eqref{eq:CP}, goes back to~\cite{EPDroyal, Davis56}. The study of the singular Cauchy problem for the EPD equation with inhomogeneous term~$f=f(t,x)$ goes back to~\cite{Y67}, whereas global-in-time existence results for small data solutions to the semilinear problem with some class of nonlinearities~$f=f(t,x,v)$ have been recently obtained in~\cite{Zhang18} for~$\mu\in(-1,0)$ (see~\cite{Uesaka94} for the case of absorbing nonlinearity~$f(u)=-v^3$ in space dimension~$n=3$). For additional references and for applications of EPD equations to gas dynamics, hydrodynamics, mechanics, elasticity and plasticity and so on, we address the reader to~\cite{SS17}.

\subsection{The criticality of the dissipative term}

The term~$\mu t^{-1}u_t$ in~\eqref{eq:CPsing} and~\eqref{eq:CP} represents a \emph{critical dissipation} acting on the wave model. It is critical in the sense that its scaling makes relevant the size of the parameter~$\mu$ in describing the transition from a wave model to a heat model. For instance, according to the size of~$\mu$ with respect to the space dimension~$n$, the energy estimates for the solution to the linear problem
\begin{equation}\label{eq:CPlinear}
\begin{cases}
v_{tt}-\triangle v+ \dfrac\mu{t}\, v_t= 0, & t\geq s>0, \ x\in\R^n\,,\\
v(s,x)=0\,, \quad v_t(s,x)=v_1(x)\,,
\end{cases}
\end{equation}
have different decay profiles. In particular~\cite{W04}, %\cite[Theorem 3.4]{W04}),
\begin{equation}\label{eq:energyest2}
E(t) = \frac12\,\|v_t(t,\cdot)\|_{L^2}^2 + \frac12\,\|\nabla v(t,\cdot)\|_{L^2}^2 \leq C\,t^{-\min\{\mu,2\}}\|v_1\|_{L^2}^2,
\end{equation}
where~$C=C(s)>0$ for energy solutions to~\eqref{eq:CPlinear}. In this manuscript, we call energy solution any (weak) solution in the space~$\mathcal C([0,\infty),H^1)\cap\mathcal C^1([0,\infty),L^2)$, so that the energy functional is well-defined and continuous.

 %, and, more in general~\cite{W05}, %[Theorem 2.9]{W05},
%%
%\[ \|\partial_t^k\partial_x^\alpha v(t,\cdot)\|\leq C\,t^{-\min\left\{\frac\mu2,k+|\alpha|\right\}}\,\|v_1\|_{H^{k+|\alpha|-1}},\quad k+|\alpha|\geq1.\]
%%
%We stress that the dependence of the decay profiles from the size of the parameter~$\mu$ is a phenomenon which depends on the fact that problem~\eqref{eq:CPlinear} is nonsingular, i.e., $s>0$. For the singular problem, the situation is pretty different, see Section~\ref{sec:sing}.
The case of sufficiently large~$\mu$ in~\eqref{eq:CPlinear} %(possibly, with respect to the space dimension~$n$)
is analogous to the case of so-called \emph{effective dissipation} \cite{W07}, in which the asymptotic profile of a wave equation with time-dependent dissipation
\begin{equation}\label{eq:CPblinear}
\begin{cases}
v_{tt}-\triangle v+ b(t)\, v_t= 0, & t\geq s, \ x\in\R^n\,,\\
v(s,x)=0\,, \quad v_t(s,x)=v_1(x)\,,
\end{cases}
\end{equation}
is described by the solution to the corresponding heat equation~$-\triangle v+ b(t)\, v_t= 0$, namely,
\[ v(t,x)= v(s,x)\,(4\pi B(t,s))^{-\frac{n}2}\,e^{-\frac{|x|^2}{4B(t,s)}}\,, \quad\text{where}\ B(t,s)=\int_s^t \frac1{b(\tau)}\,d\tau. \]
For the case~$b=1$, this ``diffusion phenomenon'' has been widely investigated in~\cite{HM, HT, MN03, N03}.

\subsection{The influence of the dissipation on the critical exponent}

For a class of coefficients~$b(t)$ as the one considered in~\cite{W07} (in particular, $b(t)=\mu t^\kappa$, with~$\mu>0$ and~$\kappa\in(-1,1)$), the diffusion phenomenon reflects on the fact that the critical exponent~$p_\crit$ for the corresponding semilinear problem
\begin{equation}\label{eq:CPb}
\begin{cases}
u_{tt}-\triangle u+ b(t)\, u_t= |u|^p, & t\geq s, \ x\in\R^n\,,\\
u(s,x)=0\,, \quad u_t(s,x)=u_1(x)\,,
\end{cases}
\end{equation}
is Fujita exponent~$1+2/n$ (see~\cite{DAL13, DALR13, LNZ12, N11, W17}), the same critical exponent appearing for the semilinear heat equation. In the case~$b=1$, the fact that the critical exponent was Fujita exponent~$1+2/n$ was proved by G. Todorova and B. Yordanov~\cite{TY01} in any space dimension~$n\geq1$ (see~\cite{Z} for the blow-up in the critical case), see also~\cite{IMN04, N04}. A previous existence result in space dimension~$n=1,2$ was proved by A. Matsumura~\cite{Matsu}.

By critical exponent~$p_\crit$ we mean that global-in-time small data solutions exist for~$p>p_\crit$ and do not exist for~$p\in(1,p_\crit]$, under suitable data sign assumptions. The study of these kind of problems has been originated by the pioneering paper of H. Fujita~\cite{F66} about the semilinear heat equation. In general, nonlinear phenomena may break the boot-strap argument which allows to prolong local-in-time solutions. H. Fujita investigated how this occurrence is prevented for sufficiently small initial data if, and only if, the power nonlinearity is larger than a given threshold exponent.

The fact that the diffusion phenomenon, i.e., the analogy with the corresponding heat equation, depends on the size of the dissipation parameter~$\mu$ in~\eqref{eq:CPlinear} suggests that the size of~$\mu$ has a direct influence on the critical exponent for~\eqref{eq:CP}. For~$\mu\geq n+2$, the author~\cite{DA15} used weighted energy estimates similar to the ones employed in~\cite{IT05}, to prove that the critical exponent is also Fujita exponent~$1+2/n$.

\subsection{The transition to a new critical exponent}

In~\cite{DALR15}, the author, S. Lucente and M. Reissig studied the special case~$\mu=2$ and showed that the critical exponent for~\eqref{eq:CP} was given by
\[ p_\crit=\max\{p_\Str(n+2),1+2/n\}= \begin{cases}
3 & \text{if~$n=1$,}\\
p_\Str(n+2) & \text{if~$n\geq2$,}
\end{cases} \]
where~$p_\Str(k)$ is the critical exponent conjectured by W.A. Strauss~\cite{Strauss} (see also~\cite{S89}) for the semilinear wave equation (namely, $\mu=0$ in~\eqref{eq:CP}), i.e., the solution to
\[ \frac{k-1}2(p-1)-1-\frac1p=0. \]
%
%or, equivalently,
%%
%\[ (k-1)(p-1)^2 + (k-3)(p-1) -4 =0. \]
%
The conjecture of W.A. Strauss for the semilinear wave equation was supported by the result obtained in the pioneering paper by F. John~\cite{J} in space dimension~$n=3$ and by the blow-up result obtained by R.T. Glassey~\cite{G} in space dimension~$n=2$. It was later proved in a series of papers, see~\cite{JZ, Sc, Si, YZ06} for blow-up results, and~\cite{A, Georgiev, GLS, G2, K5, LS96, T, Zhou} for existence results.

In~\cite{DALR15}, the blow-up in finite time for the solution to~\eqref{eq:CP} with~$\mu=2$ is proved in any space dimension~$n\geq1$ for~$1<p\leq p_\crit$, and the global-in-time existence of small data solutions is proved for~$p>p_\crit$ in space dimension~$n=2,3$. This latter result is extended in any space dimension~$n\geq5$, odd (see~\cite{DAL15}) and in any space dimension~$n\geq4$, even (see~\cite{Pal19}).

The nature of the competition between the shifted Strauss exponent and the Fujita exponent, leaded to the conjecture that the critical exponent is
\begin{equation}\label{eq:pcrit}
p_\crit=\max\{p_\Str(n+\mu),1+2/n\}=\begin{cases}
1+2/n & \text{if~$\mu\geq\bar \mu$,}\\
p_\Str(n+\mu) & \text{if~$\mu\leq\bar \mu$,}
\end{cases}\end{equation}
where the threshold value~$\bar\mu$, which corresponds to the solution to~$p_\Str(n+\mu)=1+2/n$, is given by
\begin{equation}\label{eq:barmu}
\bar \mu=n-1+\frac4{n+2}\,.
\end{equation}
By ``shifted'' Strauss exponent we mean that the space dimension~$n$ is shifted by a quantity equal to the size of the parameter~$\mu$, in the computation of the exponent.

M. Ikeda and M. Sobajima~\cite{IkedaSob17} obtained blow-up in finite time for~$f=|u|^p$ if $1<p\leq p_\Str(n+\mu)$ for suitable data, when~$\mu\leq\bar\mu$ (see also~\cite{TuLin19}), so proving the nonexistence side of the conjecture. Their result extended the blow-up result obtained for~$1<p\leq p_\Str(n+2\mu)$ by N.- A. Lai, H. Takamura, K. Wakasa in~\cite{LaiTakWak17}.

Overall, there has been a growing interest in recent years on the problems originated by the study of~\eqref{eq:CP} in~\cite{DA15, DALR15}. For lifespan estimates of the local-in-time solutions we address the reader to~\cite{IKTW, KS, KTW, Waka16}. %For models with different nonlinearities, we address the reader to~\cite{ITK, PT+} (Pal, Pal Lucente)
A closely related model is the semilinear wave equation with scale-invariant mass and dissipation, namely, a term~$mt^{-2}u$ is added into~\eqref{eq:CP}; for the studies on this topic, we address the reader to~\cite{CP19, DAP19, NPR17, Pal18, Pal19t, Pal20, PR18, PR19} and the references therein.

\subsection{Result for the one-dimensional case}

In this paragraph, we consider singular problem~\eqref{eq:CPsing} and regular problem~\eqref{eq:CP} in the one-dimensional case.
\begin{theorem}\label{thm:main1}
Let~$n=1$, $\mu>0$ and~$p>p_\crit=\max\{p_\Str(1+\mu),3\}$. Then there exists~$\eps>0$ such that for any
\begin{equation}\label{eq:datasing}
u_0\in L^1\cap H^1, \quad \text{with}\ A=\|u_0\|_{L^1}+\|u_0\|_{H^1}\leq\eps,
\end{equation}
there exists a unique $u\in\mathcal C([0,\infty),H^1)\cap\mathcal C^1([0,\infty),L^2)$, global-in-time energy solution to~\eqref{eq:CPsing}, and for any
\begin{equation}\label{eq:data}
u_1\in L^1\cap L^2, \quad \text{with}\ A=\|u_1\|_{L^1}+\|u_1\|_{L^2}\leq\eps,
\end{equation}
there exists a unique $u\in\mathcal C([t_0,\infty),H^1)\cap\mathcal C^1([t_0,\infty),L^2)$, global-in-time energy solution to~\eqref{eq:CP}. Moreover, for any~$\delta>0$, the energy estimate
\begin{equation}\label{eq:energyestn1}
E(t) = \frac12\,\|u_t(t,\cdot)\|_{L^2}^2 + \frac12\,\|u_x(t,\cdot)\|_{L^2}^2
    \leq C\,A^2\times\begin{cases}
(1+t)^{-3} & \text{if~$\mu>3$},\\
(1+t)^{\delta-3}\, & \text{if~$\mu=3$,}\\
(1+t)^{-\mu} & \text{if~$0<\mu<3$ with~$\mu\neq1$,}\\
(1+t)^{-1}(1+\log(1+t))^2 & \text{if~$\mu=1$,}
\end{cases}
\end{equation}
holds, and we have the decay estimate
\begin{equation}\label{eq:decayn1}
\|u(t,\cdot)\|_{L^q} \leq C\,A\times \begin{cases}
(1+t)^{(1-\mu)_+-1+\frac1q} & \text{if~$1-1/q<\max\{\mu,2-\mu\}$,} \\
(1+t)^{\delta-\frac\mu2} & \text{if~$1-1/q\geq \max\{\mu,2-\mu\}$,}
\end{cases}
\end{equation}
where~$C>0$, for any~$q\in[3,\infty)$.
\end{theorem}
Theorem~\ref{thm:main1} proves the conjecture that the critical exponent for~\eqref{eq:CP} is given by~\eqref{eq:pcrit} in space dimension~$n=1$ and extends the validity of this result to the singular problem~\eqref{eq:CPsing}.
\begin{remark}\label{rem:n1}
Theorem~\ref{thm:main1} provides the global-in-time existence of energy solutions to~\eqref{eq:CPsing} for~$p>3$ if~$\mu\geq4/3$ and for
\[ p>p_\Str(1+\mu)= 1 + \frac{2-\mu + \sqrt{\mu^2+12\mu+4}}{2\mu}, \]
if~$\mu\in(0,4/3]$. As expected, $p_\Str(1+\mu)\nearrow\infty$ as~$\mu\searrow0$ ($p_\Str(1+\mu)\sim1+2/\mu$ for small~$\mu$), consistently with the blow-up result for the semilinear wave equation without damping.
\end{remark}
\begin{remark}\label{rem:optimaln1}
The exponent~$3$ in Theorem~\ref{thm:main1} is the Fujita exponent and it is sharp, in the sense that no global-in-time solutions exist if~$f=|u|^p$ with~$p\in(1,3]$, for suitable sign assumption on the initial datum, see Theorem 1.1 in~\cite{DAL13}. The exponent~$p_\Str(1+\mu)$ in Theorem~\ref{thm:main1} is a shifted Strauss exponent. The nonexistence of global-in-time solutions for~$f=|u|^p$ with $3<p\leq p_\Str(1+\mu)$ has been recently proved in~\cite{IkedaSob17}.
\end{remark}
In Section~\ref{sec:proof1} we also discuss the analogous of Theorem~\ref{thm:main1} when a nonlinearity~$t^{-\alpha}f(u)$ is considered. On the one hand, this generalization is of interest for the possibility to obtain, by a change of variable, results for semilinear generalized Tricomi equations
\[ w_{tt}- t^{2\ell}\,w_{xx} = f(w), \]
setting~$\mu=\ell/(\ell+1)$ and~$\alpha=2\mu$ (see Section~\ref{sec:Tricomi}). On the other hand, this generalization provides more insights about how the critical exponent~$p_\crit$ depends on~$\mu$ and~$\alpha$.

\subsection{Result in the multidimensional case}

In space dimension~$n=2,3,4,5$, we may prove the existence of global-in-time energy solutions, for small data in~$L^1$ and in the energy space, to the regular Cauchy problem~\eqref{eq:CP}, when~$\mu\geq n$ and~$p>1+2/n$.
\begin{theorem}\label{thm:main2}
Let~$n=2,3,4,5$ and~$\mu\geq n$. Assume that~$p>p_\crit=1+2/n$, and that~$p\leq 1+2/(n-2)$ if~$n\geq3$. Then there exists~$\eps>0$ such that for any initial data as in~\eqref{eq:data}, there exists a unique global-in-time energy solution~$u$ to~\eqref{eq:CP}, where $u\in\mathcal C([t_0,\infty),H^1)\cap\mathcal C^1([t_0,\infty),L^2)$, if~$n=2$, and~$u\in\mathcal C([t_0,\infty),H^1)\cap\mathcal C^1([t_0,\infty),L^2)\cap L^\infty([t_0,\infty),L^{1+\frac2n})$, if~$n=3,4,5$. Moreover, the energy estimate
\begin{align}\nonumber
E(t)
    & = \frac12\,\|u_t(t,\cdot)\|_{L^2}^2 + \frac12\,\|\nabla u(t,\cdot)\|_{L^2}^2\\
    \label{eq:energyestn2}
    & \leq C\,\big(\|u_1\|_{L^1}^2+\|u_1\|_{L^2}^2\big)\times\begin{cases}
t^{-n-2} & \text{if~$\mu>n+2$},\\
t^{-n-2}\,(1+\log (t/t_0)) & \text{if~$\mu=n+2$,}\\
t^{-\mu} & \text{if~$n\leq\mu<n+2$,}
\end{cases}
\end{align}
holds, and for any~$q\in[p_\crit,2+4/(n-1)]$ and~$\delta>0$, the following decay estimate holds:
\begin{equation}\label{eq:decay2}
\|u(t,\cdot)\|_{L^q} \leq C\,\begin{cases}
t^{-n\left(1-\frac1q\right)}\,\big(\|u_1\|_{L^1}+\|u_1\|_{L^2}\big) & \text{if~$\mu>n+1-2/q$,}\\
t^{\delta+(n-1)\left(\frac12-\frac1q\right)-\frac\mu2}\,\big(\|u_1\|_{L^1}+\|u_1\|_{L^2}\big) & \text{if~$\mu\leq n+1-2/q$,}
\end{cases}
\end{equation}
where~$C=C(t_0)>0$.
\end{theorem}
The assumption~$\mu\geq n$ in Theorem~\ref{thm:main2} means that we have the decay rate~$t^{-n\left(1-\frac1q\right)}$ for any~$q\in[p_\crit,2)$, in space dimension~$n=3,4,5$. This will be a crucial property in our proof of Theorem~\ref{thm:main2}.
\begin{remark}
In space dimension~$n=2$, the result is easily extended to the singular problem~\eqref{eq:CPsing}, following as in the proof of Theorem~\ref{thm:main1}. Moreover, in space dimension~$n=2$, the threshold condition~$\mu\geq2$ is also sharp, since~$\bar\mu(2)=2$ in~\eqref{eq:barmu}. That is, for~$\mu\in(0,2)$, the blow-up in finite time occurs for~$f=|u|^p$ with $2<p\leq p_\Str(2+\mu)$ for suitable data, when~$\mu<2$.
\end{remark}
Finally, we consider problem~\eqref{eq:CP} with the assumption that initial data are only small in~$L^2$. As first noticed in~\cite{LN92} for parabolic problems, initial data which do not decay sufficiently fast at infinity, in particular are not in~$L^1$, modify the critical exponent, even if they are taken pointwisely small. In particular, if~$L^1$ smallness of the data is replaced by~$L^2$ smallness ``only'', the critical exponent switches from~$1+2/n$ to~$1+4/n$. In the following Theorem~\ref{thm:L2}, we prove the global-in-time existence of weak (Sobolev) solutions in any space dimension~$n\geq3$ for~$p\geq 1+4/n$, and~$p\leq p_\conf = 1+4/(n-1)$, for~$\mu\geq 2n/(n+3)$, under the assumption of small data in~$L^2$. Here~$p_\conf$ is the \emph{conformal critical exponent} for semilinear waves (see, for instance, \cite{LS95}). For the sake of brevity, we omit the study of the easier case of existence of global-in-time solutions in space dimension~$n=1,2$.
\begin{theorem}\label{thm:L2}
Let~$n\geq3$ and $\mu\geq\bar\mu$, where
\[ \bar\mu = \frac{2n}{n+3}. \]
Assume that~$p\geq p_\crit$, where~$p_\crit=1+4/n$, and that~$p\leq p_\conf = 1+4/(n-1)$. Then there exists~$\eps>0$ such that for any
\begin{equation}\label{eq:data2}
u_1\in L^2, \quad \text{with}\ \|u_1\|_{L^2}\leq\eps,
\end{equation}
there exists a unique $u\in L^\infty([t_0,\infty),L^{q_0}\cap L^{q_1})$, global-in-time energy solution to~\eqref{eq:CP}, where
\begin{equation}\label{eq:q0q1}
q_0=2+\frac4{n+1}, \quad q_1=p_\conf+1=2+\frac4{n-1}\,.
\end{equation}
Moreover, for any~$q\in[q_0,q_1]$, the following decay estimate holds:
\begin{equation}\label{eq:decayL2}
\|u(t,\cdot)\|_{L^q} \leq C\,t^{-\frac12\min\left\{n\left(1-\frac2q\right),\mu\right\}}\,\|u_1\|_{L^2},
\end{equation}
where~$C=C(t_0)>0$, exception given for the case~$q=n=3$ when~$\mu=1$. If~$\mu=1$ and~$q=n=3$ and~$\mu=1$, estimate~\eqref{eq:decayL2} is replaced by
\[ \|u(t,\cdot)\|_{L^3} \leq C\,t^{-\frac12}(1+\log (t/t_0))\,\|u_1\|_{L^2}, \]
when~$q=3$.
\end{theorem}
The existence exponent~$1+4/n$ is also critical, in the sense that one may easily follow the proof of Theorem 1.1 in~\cite{DAL13}, adding the condition~$u_1\geq \eps\,|x|^{-\frac{n}2}(\log |x|)^{-1}$ (this strategy is inspired by~\cite{MP09}), for large~$|x|$ and for some~$\eps>0$, and prove that no global-in-time solutions to~\eqref{eq:CP} exist for~$1<p<1+4/n$.

We stress that for~$p=p_\crit$, the global-in-time existence of small data solution holds, that is, the critical case belong to the existence range.
\begin{remark}
In space dimension~$n=3$, it is possible to consider also energy solutions $u\in\mathcal C([t_0,\infty),H^1)\cap\mathcal C^1([t_0,\infty),L^2)$ in Theorem~\ref{thm:L2}. The same is possible in space dimension~$n=4$ if~$p=2$.
\end{remark}
\begin{remark}
In Theorem~\ref{thm:L2}, we looked for weak solutions in $L^\infty([t_0,\infty),L^{q_0}\cap L^{q_1})$, but there is no big difference if we look for weak solutions in $L^\infty([t_0,\infty),\dot H^{\kappa_0}\cap \dot H^{\kappa_1})$, with
\[ \kappa_0=n\left(\frac12-\frac1{q_0}\right)=\frac{n}{n+3}, \qquad \kappa_1=n\left(\frac12-\frac1{q_1}\right)=\frac{n}{n+1}\,. \]
\end{remark}

\subsection{Notation}

In this paper we use the following notation.

We denote by~$u(t,x)$ functions depending on the time variable~$t\in I$, with~$I$ interval in~$\R$, and on the space variable~$x\in\R^n$, and we denote by~$\mathfrak{F}$ the Fourier transform acting on the space variable~$x$, in the appropriate functional sense. By~$\Delta$ we denote the Laplace operator~$\sum_{j=1}^n\partial_{x_j}^2$, and by~$\nabla u$ the gradient vector~$(\partial_{x_j}u)_{j=1,\ldots,n}$.

By~$L^q=L^q(\R^n)$, $1\leq q<\infty$, we denote the usual Lebesgue space of measurable functions with~$|u|^q$ integrable with respect to the Lebesgue measure~$dx$ of~$\R^n$. We denote by
\[ \|f\|_{L^q} = \Big(\int_{\R^n} |f(x)|\,dx\Big)^{\frac1q} \]
its norm (a.e. equal functions are identified, as usual). For functions~$u(t,x)$ we denote by~$\|u(t,\cdot)\|_{L^q}$ the~$L^q$ norm of~$u(t,\cdot)$, for a given~$t$. By~$H^1$ we denote the space of~$L^2$ functions with weak gradient in~$L^2$, equipped with norm~$\|f\|_{L^2}+\|\nabla f\|_{L^2}$. By~$\mathcal C(I,H^1)\cap \mathcal C^1(I,L^2)$ we generally denote the space of energy solutions, that is, the maps~$t\mapsto u(t,\cdot)$ and~$t\mapsto u_t(t,\cdot)$ are continuous from~$I$ to, respectively, $H^1$ or~$L^2$. By~$L^\infty(I,L^p)$ we denote the space of functions with~$\|u(t,\cdot)\|_{L^p}$ uniformly bounded, for a.e.~$t\in I$.

We say that~$m$ is a multiplier in~$M_r^q$, for some~$1\leq r\leq q\leq\infty$ if for any~$f\in L^r$ it holds~$T_mf = \mathfrak{F}^{-1}(m\hat f) \in L^q$. We denote
\begin{equation}\label{eq:multnorm}
\|m\|_{M_r^q} = \sup_{\|f\|_{L^r}=1} \|T_mf\|_{L^q}.
\end{equation}

\section{Estimates for the linear problem}\label{sec:linear}

The Euler-Poisson-Darboux equation~\eqref{eq:EPD} is not invariant by time-translation, due to the time-dependent coefficient~$\mu t^{-1}$ in front of~$u_t$. For this reason, we study the regular linear Cauchy problem~\eqref{eq:CPlinear} for~$t\geq s$, with starting time~$s>0$, in view of the application of Duhamel's principle to both the inhomogeneous singular and regular Cauchy problems. The dependence of the obtained estimates on the parameter~$s$ plays a crucial role in the contraction argument employed to prove the existence of global-in-time solutions: a precise evaluation of the dependence on the parameter~$s$ in the estimates is essential to ``catch the critical exponent'' in the application to the semilinear problem.

The following definition is related to the $L^r-L^q$ estimates, $1\leq r\leq q\leq \infty$, for wave-type multipliers~$\xii^{-k}e^{i\xii}$, see Lemma~\ref{lem:Miy}. The definition plays a fundamental role in the estimates for the EPD equation, due to the possibility to subdivide the solution to~\eqref{eq:CPlinear} in wave-type terms at high frequencies.
\begin{Def}
For any~$1\leq r\leq q\leq\infty$, we define
\begin{equation}\label{eq:drq}
d(r,q) = (n-1)\left(\frac1{\min\{r,q'\}}-\frac12\right) + \frac1r-\frac1q = \begin{cases}
    \displaystyle \frac{n}r-\frac{n-1}2-\frac1q & \text{if~$r\leq q'$,}\\
    \displaystyle \frac{n-1}2+\frac1r-\frac{n}q & \text{if~$r\geq q'$.}
    \end{cases}
\end{equation}
\end{Def}
The interplay between a scaling-related decay rate~$t^{-n\left(\frac1r-\frac1q\right)}$ and a loss of decay rate~$t^{d(r,q)}$ will often appear in the following, so it is convenient to notice that
\[ -n\left(\frac1{r}-\frac1q\right) + d(r,q) = (n-1)\left(\frac12-\frac1{\max\{r,q'\}}\right) = \begin{cases}
    \displaystyle (n-1)\left(\frac1q-\frac12\right) & \text{if~$r\leq q'$,}\\
    \displaystyle (n-1)\left(\frac12-\frac1r\right) & \text{if~$r\geq q'$.}
    \end{cases} \]
Our main result for the linear regular problem~\eqref{eq:CPlinear} is the following.
\begin{theorem}\label{thm:main}
Let~$\mu\in\R$. Let~$n=1,2,3$ and~$q\in(1,\infty)$, or~$n\geq4$ and
\begin{equation}\label{eq:qrange}
\frac{2(n-1)}{n+1} \leq q \leq \frac{2(n-1)}{n-3}\,.
\end{equation}
Fix~$r_1,r_2\in(1,q]$. Assume that~$d(r_2,q)\leq1$, where~$d$ is defined in~\eqref{eq:drq}. Then the solution to~\eqref{eq:CPlinear} verifies the following~$(L^{r_1}\cap L^{r_2})-L^q$ estimate:
\begin{equation}\label{eq:gendecay}
\begin{split}
\|v(t,\cdot)\|_{L^q}
     & \leq C_1\,s^{\min\{1,\mu\}}\,t^{(1-\mu)_+-n\left(\frac1{r_1}-\frac1q\right)}\,(t/s)^{\left(d(r_1,q)-\frac{\max\{\mu,2-\mu\}}2\right)_+}\, \|v_1\|_{L^{r_1}} \\
     & \qquad + C_2\,s\,t^{-n\left(\frac1{r_2}-\frac1q\right)}\,(t/s)^{d(r_2,q)-\frac\mu2}\,\|v_1\|_{L^{r_2}} \,,
\end{split}\end{equation}
for some~$C>0$, independent of~$s,t$, if~$\mu\neq1$. If~$\mu=1$, estimate~\eqref{eq:gendecay} remains valid, replacing~$\|v_1\|_{L^{r_1}}$ by~$(1+\log(t/s))\,\|v_1\|_{L^{r_1}}$.

For any~$\eps>0$, the above result remains valid, for~$C_j=C_j(\eps)$, if we replace~$d(r_j,q)$ by~$d(1,q)+\eps$ whenever~$r_j=1$.
\end{theorem}
Classic $L^r-L^q$ estimates, $1\leq r\leq q<\infty$ are obtained by Theorem~\ref{thm:main}, setting~$r_1=r_2$.
\begin{corollary}\label{cor:main}
Let~$\mu\in\R$. Let~$n=1,2,3$ and~$q\in(1,\infty)$, or~$n\geq4$ and~$q$ as in~\eqref{eq:qrange}. Fix~$r\in(1,q]$ such that~$d(r,q)\leq1$.

If~$\mu\neq1$ and~$d(r,q)\leq \max\{ \mu,2-\mu\}/2$, then the solution to~\eqref{eq:CPlinear} verifies the following~$L^r-L^q$ estimate:
\begin{equation}\label{eq:rqdecaygood}
\|v(t,\cdot)\|_{L^q} \leq C\,s^{\min\{1,\mu\}}\,t^{(1-\mu)_+-n\left(\frac1{r}-\frac1q\right)}\, \|v_1\|_{L^{r}} \,,
\end{equation}
%
%\begin{equation}\label{eq:rqdecaygood}
%\|v(t,\cdot)\|_{L^q} \leq C\,s\,t^{-n\left(\frac1{r}-\frac1q\right)}\, \|v_1\|_{L^{r}} \,,
%\end{equation}
%
for some~$C>0$, independent of~$s,t$, if~$\mu\neq1$. If~$\mu\neq1$ and~$d(r,q)>\max\{\mu,2-\mu\}/2$, then the solution to~\eqref{eq:CPlinear} verifies the following~$L^r-L^q$ estimate:
\begin{equation}\label{eq:rqdecaybad}
\|v(t,\cdot)\|_{L^q} \leq C\,s^{1-d(r,q)+\frac\mu2}\,t^{-n\left(\frac1{r}-\frac1q\right)+d(r,q)-\frac\mu2}\, \|v_1\|_{L^{r}}\,.
\end{equation}
If~$\mu=1$, estimates~\eqref{eq:rqdecaygood} and~\eqref{eq:rqdecaybad} remain valid, replacing~$\|v_1\|_{L^{r}}$ by~$(1+\log(t/s))\,\|v_1\|_{L^{r}}$. %
%If~$\mu\in(0,1)$ and~$d(r,q)\leq 1-\mu/2$, %or~$\mu\in(-\infty,0)$ and~$d(r,q)\leq 1$,
%then the solution to~\eqref{eq:CPlinear} verifies the following~$L^r-L^q$ estimate:
%
%\begin{equation}\label{eq:rqdecaygoodsharp}
%\|v(t,\cdot)\|_{L^q} \leq C\,s^{\mu}\,t^{1-\mu-n\left(\frac1{r}-\frac1q\right)}\, \|v_1\|_{L^{r}} \,,
%\end{equation}
%
For any~$\eps>0$, the above results remains valid, for~$C=C(\eps)$, if we replace~$d(r,q)$ by~$d(1,q)+\eps$ when~$r=1$.
\end{corollary}
\begin{remark}
The condition~$d(r,q)\leq1$ (or~$d(1,q)<1$ if~$r=1$) in Corollary~\ref{cor:main} is necessary and sufficient to obtain $L^r-L^q$ estimates for the wave equation with no damping (see later, Lemma~\ref{lem:Miy}). In particular, assumption~\eqref{eq:qrange} is equivalent to~$d(q,q)\leq1$, the condition for the~$L^q$ boundedness of the solution operator for the wave equation without damping~\cite{Peral}.

Setting~$\mu=0$ in~\eqref{eq:rqdecaygood}, the estimate is consistent with the classical $L^r-L^q$ estimate for the wave equation:
\[ \|v(t,\cdot)\|_{L^q} \leq C\,t^{1-n\left(\frac1{r}-\frac1q\right)}\, \|v_1\|_{L^{r}}\,. \]
%
%This relation is due to the fact that at high frequencies, the solution to~\eqref{eq:CPlinear} may be subdivided in terms with the wave behavior.
However, when~$\mu>0$ the presence of the damping term has a twofold benefit on the decay estimates for~\eqref{eq:CPlinear}. On the one hand, it produces additional decay rate. On the other hand, this decay rate may be enhanced replacing $L^r-L^q$ estimates by $(L^{r_1}\cap L^{r_2})-L^q$ estimates. More precisely, the decay rate obtained by Theorem~\ref{thm:main} is better than the one in Corollary~\ref{cor:main}, if both $d(r_1,q)$ and~$\max\{\mu,2-\mu\}/2$ are greater than~$1$, in view of the bound~$d(r_2,q)\leq1$. In this way, benefits on the decay rate may be obtained mixing the~$L^{r_1}$ regularity for the data at intermediate frequencies, with the~$L^{r_2}$ regularity for the data at high frequencies. This interplay plays a crucial role in obtaining sharp results for the semilinear problem~\eqref{eq:CP} when~$n\geq2$.
\end{remark}
Corollary~\ref{cor:main} is sufficient to prove Theorem~\ref{thm:main1}, in view of the fact that~$d(1,q)<1$ for any~$q\in(1,\infty)$. However, in space dimension~$n=2$, the bound~$d(1,q)<1$ is violated for~$q\geq2$. Moreover, in space dimension~$n\geq3$, the condition~$d(1,q)<1$ holds for no~$q$. For this reason, we take advantage of the $(L^{r_1}\cap L^{r_2})-L^q$ estimates in Theorem~\ref{thm:main}, in which we fix~$d(r_2,q)=1$ and we take~$r_1=1$.
\begin{corollary}\label{cor:main2}
Let~$\mu\geq2$. Let~$n=2$ and~$q\in(2,6]$, or~$n=3$ and~$q\in(1,4]$, or~$n\geq4$ and
\begin{equation}\label{eq:qrangeprime}
\frac{2(n-1)}{n+1}\leq q\leq \frac{2(n+1)}{n-1}.
\end{equation}
Then there exists~$r_2\in(1,\min\{q,q'\}]$ such that~$d(r_2,q)=1$ and the solution to~\eqref{eq:CPlinear} verifies the following $(L^1\cap L^{r_2})-L^q$ decay estimate
\begin{equation}\label{eq:gendecaygood}
\|v(t,\cdot)\|_{L^q} \leq C\,s\,t^{-n\left(1-\frac1q\right)}\, \big(\|v_1\|_{L^1} + s^{\frac{n-1}2-\frac1q}\,\|v_1\|_{L^{r_2}}\big) \,,
\end{equation}
if~$\mu>n+1-2/q$, and for any~$\eps>0$ verifies the $(L^1\cap L^{r_2})-L^q$ estimate
\begin{equation}\label{eq:gendecaybad}
\|v(t,\cdot)\|_{L^q} \leq C(\eps)\,s^{\frac\mu2-\eps}\,t^{\eps-(n-1)\left(\frac12-\frac1q\right)-\frac\mu2}\,\big(s^{-\frac{n-1}2+\frac1q}\, \|v_1\|_{L^1} + \|v_1\|_{L^{r_2}}\big) \,,
\end{equation}
if~$\mu\leq n+1-2/q$.
\end{corollary}
For the ease of reading, we provide the straightforward proof of Corollary~\ref{cor:main2}.
\begin{proof}
First of all, we notice that~$d(q,q)\leq1$, since~\eqref{eq:qrangeprime} implies~\eqref{eq:qrange}, if~$n\geq4$. On the other hand, the right-hand bound~$q\leq2(n+1)/(n-1)$ guarantees that~$d(q',q)\leq1$ when~$q\geq2$. As a consequence, there exists~$r_2\in(1,\min\{q,q'\}]$ such that~$d(r_2,q)=1$.

Since~$d(r_2,q)=1$ and~$r_2\leq q'$, we may replace
\[ s\,t^{-n\left(\frac1{r_2}-\frac1q\right)}(t/s)^{d(r_2,q)-\frac\mu2} = s^{\frac\mu2}\,t^{-(n-1)\left(\frac12-\frac1q\right)-\frac\mu2}, \]
in~\eqref{eq:gendecay}. Moreover, if~$\mu>n+1-2/q$, we may estimate
\[ s^{\frac\mu2}\,t^{-(n-1)\left(\frac12-\frac1q\right)-\frac\mu2} \leq s^{\frac{n+1}2-\frac1q}\,t^{-n\left(1-\frac1q\right)}. \]
Therefore, by~\eqref{eq:gendecay} we derive~\eqref{eq:gendecaygood}. On the other hand, if~$2\leq\mu\leq n+1-2/q$, we immediately obtain~\eqref{eq:gendecaybad}, using~$(s/t)^\eps\leq1$.
\end{proof}
\begin{remark}\label{rem:change}
It is sufficient to prove Theorem~\ref{thm:main} for~$\mu\geq1$. Indeed, let~$\mu\in(-\infty,1)$ in~\eqref{eq:CPlinear}. If we define
\begin{equation}\label{eq:musharp}
v^\sharp(t,x)=t^{\mu-1}\,v(t,x)\,, \qquad \text{and} \quad \mu^\sharp=2-\mu\,,
\end{equation}
then Cauchy problem~\eqref{eq:CPlinear} becomes
\begin{equation}\label{eq:CPlinearsharp}
\begin{cases}
v_{tt}^\sharp- \triangle v^\sharp + \dfrac{\mu^\sharp}{t}\,v_t^\sharp = 0, & t>s\,, \ x\in\R^n\,,\\
v^\sharp(s,x)=0, \quad v_t^\sharp(s,x)=s^{1-\mu}\,v_1(x).
\end{cases}
\end{equation}
Applying Theorem~\ref{thm:main} to~\eqref{eq:CPlinearsharp} with~$\mu^\sharp>1$, we obtain the statement of Theorem~\ref{thm:main} for~$\mu<1$.
\end{remark}

\subsection{The fundamental solution to~\eqref{eq:CPlinear}}

For the ease of reading, we divide the proof of Theorem~\ref{thm:main} in steps. We mention that some $L^r-L^q$ estimates have been previously obtained in~\cite{W05}, but we need a complete~$(r,q)$ range of estimates, with a dependence on the parameter~$s$, to apply them to the semilinear problems~\eqref{eq:CPsing} and~\eqref{eq:CP}.

Let~$K(t,s)$ be the fundamental solution to~\eqref{eq:CPlinear}. The Fourier transform of~$K(t,s)$ with respect to the space variable solves the problem
\begin{equation}\label{eq:CPK}
\begin{cases}
\hat K_{tt}+\xii^2\hat K + \dfrac\mu{t}\,\hat K_t= 0, & t>s,\\
\hat K(s,s)=0\,, \quad \hat K_t(s,s)=1\,.
\end{cases}
\end{equation}
The equation in~\eqref{eq:CPK} is \emph{scale-invariant}, namely, if we set
\[ \tau=t\xii, \quad \sigma=s\xii,\quad w(t\xii)=\hat K(t,s),\]
we find the equivalent problem
\begin{equation}
\label{eq:scale}
\begin{cases}
w'' + w + \dfrac\mu\tau \, w' =0, \qquad \tau\geq \sigma\,,\\
w(\sigma)=0\,, \quad w'(\sigma)=\xii^{-1}\,.
\end{cases}
\end{equation}
If we put~$\nu \doteq (\mu-1)/2$ and~$y(\tau)=\tau^\nu\,w(\tau)$, then
%
%\[ y'=\nu\tau^{\nu-1}w+\tau^\nu w',\qquad y''=\nu(\nu-1)\tau^{\nu-2}w+2\nu\tau^{\nu-1}w'+\tau^\nu w'', \]
%
%so that
%
%\[ \tau^2 y'' + \tau y' + (\tau^2-\nu^2) y = \tau^\nu\big(w''+(2\nu+1)w'\tau^{-1}+w\big)=0. \]
%
from~\eqref{eq:scale} we obtain the Cauchy problem for the Bessel's differential equation of order~$\pm\nu$:
\begin{equation}\label{eq:yBess}
\begin{cases}
\tau^2 y'' + \tau y' + (\tau^2-\nu^2) y =0 \,, \qquad \tau\geq\sigma\,,\\
y(\sigma)=0,\quad y'(\sigma)=s\,\sigma^{\nu-1}.
\end{cases}
\end{equation}
We assume that~$\nu>0$ is not integer, that is, $\mu>1$ is not an odd integer. Then a system of linearly independent solutions to~\eqref{eq:yBess} is given by the pair of Bessel functions (of first kind) $J_{\pm\nu}(\tau)$, hence we put
\[ y = C_+(\sigma)\,J_\nu(\tau) + C_-(\sigma)\,J_{-\nu}(\tau).\]
The definition of Bessel functions by series is
\[ J_\rho(z)= \sum_{m=0}^\infty \frac{(-1)^m}{m!\Gamma(m+\rho+1)}\,(z/2)^{2m+\rho}. \]
We postpone the case when~$\nu$ is not an integer to Section~\ref{sec:integer}. In that case, we use a different system of linearly independent solutions to~\eqref{eq:yBess}. However, only minor changes appear, unless~$\nu=0$, that is, $\mu=1$.

Imposing the initial conditions
\[\begin{cases}
C_+ J_\nu(\sigma) + C_- J_{-\nu}(\sigma)=0,\\
C_+ J_\nu'(\sigma) + C_- J_{-\nu}'(\sigma)=s\sigma^{\nu-1},
\end{cases}\]
and recalling that the Wronskian satisfies~\cite[\textsection 3.12]{Watson}
\[ W[J_\nu,J_{-\nu}](\sigma)=J_\nu(\sigma)J_{-\nu}'(\sigma)-J_\nu'(\sigma)J_{-\nu}(\sigma)=\frac{-2\sin(\nu\pi)}{\pi\sigma}, \]
we derive
\[ y = \frac{\pi}{2\sin(\nu\pi)}\, \big( J_{-\nu}(\sigma)J_\nu(\tau)-J_\nu(\sigma)J_{-\nu}(\tau)\big)\,s\,\sigma^{\nu}, \]
so that, replacing~$\sigma=s\xii$ and~$\tau=t\xii$, we find
\[ \hat K(t,s)=\frac{\pi}{2\sin(\nu\pi)}\, \big( J_{-\nu}(s\xii)J_\nu(t\xii)-J_\nu(s\xii)J_{-\nu}(t\xii)\big)\,s^{\nu+1}\,t^{-\nu}. \]
We now want to estimate the multiplier norm~\eqref{eq:multnorm} of~$\hat K(t,s)$, depending on both~$s,t$, after localizing it.

It is clear that~$\hat K(t,s)\in L^\infty=M_2^2$, for any~$t\geq s$. Fix~$a=s/t\in(0,1]$. By homogeneity, for any~$t>0$ it holds
\begin{equation}\label{eq:scaling}
\|\hat K(t,s)\chi_j^2(t\xii)\|_{M_r^q} = s\,t^{-n\left(\frac1r-\frac1q\right)}\,\|\hat K_a\chi_j^2\|_{M_r^q},
\end{equation}
where~$\chi_j$ is a localizing function which will be fixed later, and
\begin{equation}\label{eq:Ka}
\hat K_a(\xi)=s^{-1}\,\hat K(1,a)=\frac{\pi}{2\sin(\nu\pi)}\,a^{\nu}\, \big( J_{-\nu}(a\xii)J_\nu(\xii)-J_\nu(a\xii)J_{-\nu}(\xii)\big).
\end{equation}
In order to take into account of the influence from the parameter~$a$, we fix three localizing functions~$\chi_0,\chi_1,\chi_2\in\mathcal C^\infty$, with the following properties:
\begin{itemize}
\item $\chi_0(\xi)=1$ for~$\xii\leq1/2$, and~$\chi_0$ is supported in the ``low frequencies zone'' $\{\xii\leq1\}$;
\item $\chi_2(\xi)=1$ for~$a\xii\geq2$, and $\chi_2$ is supported in the ``high frequencies zone''~$\{a\xii\geq1\}$;
\item it holds
\[ 1 = \chi_0^2 + \chi_1^2 + \chi_2^2; \]
in particular, $\chi_1$ is supported in the ``intermediate frequencies zone''~$\{1/2\leq\xii\leq 2a^{-1}\}$.
\end{itemize}
To carry over our analysis at intermediate and high frequencies, we will use the asymptotic expansion (see~\cite[\textsection 7.21]{Watson}) of the Bessel functions~$J_{\pm\nu}(z)$ for large values of~$z$,
\begin{equation}\label{eq:Jasymp}
\begin{split}
J_{\pm\nu}(z)
    & = (z\pi/2)^{-\frac12}\,\cos(z\mp\nu\pi/2-\pi/4)\sum_{m=0}^\infty (-1)^m(\nu,2m)(2z)^{-2m} \\
    &\qquad - (z\pi/2)^{-\frac12}\,\sin (z\mp\nu\pi/2-\pi/4)\sum_{m=0}^\infty (-1)^m(\nu,2m+1)(2z)^{-2m-1}.
\end{split}
\end{equation}
and the following multiplier theorem.
\begin{lemma}[see Theorem 4.2 in~\cite{Miyachising} and the references therein]\label{lem:Miy}
Let
\[ m(\xi) = \psi(\xii)\,\xii^{-k} e^{\pm i\xii} \]
where~$k>0$ and~$\psi\in\mathcal C^\infty$ vanishes near the origin and is~$1$ for large values of~$\xii$. Take~$d(r,q)$ as in~\eqref{eq:drq}. Then~$m\in M_r^q$ if, and only if, $d(r,q)\leq k$ when~$1<r\leq q<\infty$, and if, and only if, $d(r,q)<k$, when~$r=1\leq q\leq\infty$ or~$1\leq r\leq q=\infty$.
\end{lemma}
%
%To deal with multipliers localized ``intermediate frequencies'', we will use the following obvious consequence of Lemma~\ref{lem:Miy}.
%%
%\begin{remark}\label{rem:Miy}
%If~$\psi_1\in\mathcal C_c^\infty$ and vanishes near the origin, it is sufficient to apply Lemma~\ref{lem:Miy} to~$(\psi_1+\psi_2)(\xii)\,\xii^{-k} e^{\pm i\xii}$ and to~$\psi_2(\xii)\,\xii^{-k} e^{\pm i\xii}$ to derive that~$m\in M_r^q$ if $d(r,q)\leq k$ when~$1<r\leq q<\infty$, and if $d(r,q)<k$, when~$r=1\leq q\leq\infty$ or~$1\leq r\leq q=\infty$.
%\end{remark}
%
We will also make use of Mikhlin-H\"ormander multiplier theorem in its simpler form: if~$|\partial_\xi^\beta m(\xi)|\leq C\xii^{-|\beta|}$ for any~$|\beta|\leq n/2+1$, then~$m\in M_q^q$ for any~$q\in(1,\infty)$.%, and of Hardy-Littlewood-Sobolev theorem:
%%
%\[ \| \mathfrak{F}^{-1}(\xii^{-n\left(\frac1r-\frac1q\right)}\,\hat f) \|_{L^q} \leq C\,\|f\|_{L^r}, \qquad 1<r<q<\infty. \]

\subsection{$L^r-L^q$ estimates at low frequencies}

Using the definition by series of the Bessel functions, it is known that $J_{\rho}(z)\sim (z/2)^\rho/\Gamma(1+\rho)$ as~$z\to0$. As a consequence,
\[ \|\hat K_a\chi_0^2\|_{M_2^2}=\|\hat K_a\chi_0^2\|_{L^\infty}\leq C,\]
with~$C>0$ independent of~$a$. %
%Using
%%
%\[ J_{\pm\nu}(z)\sim \frac1{\Gamma(1\pm\nu)} (z/2)^{\pm\nu},\qquad \Gamma(1+\nu)\Gamma(1-\nu)=\nu\Gamma(\nu)\Gamma(1-\nu)=\frac{\nu\pi}{\sin(\nu\pi)}\,, \]
%%
%for~$\xii\leq1$, we get
%%
%\[ \hat K_a(\xi) \sim \frac1{2\nu}\,a^{\nu} \, \big( a^{-\nu} - a^\nu \big) = \frac{1 - a^{2\nu}}{2\nu}. \]
%%
%Due to~$a\leq1$ and~$\nu$ positive, we deduce~$\hat K_a\approx 1-a$, so that $\|\chi_0^2\hat K_a\|_{M_2^2}\approx 1-a$.
Using $zJ_\rho'=-\rho J_\rho + z\,J_{\rho-1}$, so that
\[ \partial_{\xi_j} J_\rho(a\xii) = (a\,\xii)\frac{\xi_j}{\xii^2}\,J_\rho'(a\xii) = \frac{\xi_j}{\xii^2}\,\big(-\rho J_\rho(a\xii) +a\xii\,  J_{\rho-1}(a\xii) \big), \]
and similarly for~$\partial_{\xi_j} J_\rho(\xii)$, iterating, we derive $|\partial_\xi^\beta \hat K_a(\xi)| \leq C\,\xii^{-|\beta|}$, with~$C$ independent of~$a$, for any~$\beta\in\N^n$. Applying Mikhlin-H\"ormander theorem, it follows that~$\hat K_a\chi_0^2\in M_q^q$ for any~$q\in(1,\infty)$, and~$\|\hat K_a\chi_0^2\|_{M_q^q}$ is uniformly bounded with respect to~$a$. Using the estimates for the derivatives of~$\hat K_a$ and recalling that~$\chi_0\in\mathcal C_c^\infty$, by standard methods, it is easy to prove the pointwise estimates (see, for instance, Lemma~8 in~\cite{DAE20})
\[ |\mathfrak{F}^{-1}(\hat K_a\chi_0^2)(x)|\leq C(1+|x|)^{-n}, \]
which guarantees, by Young inequality,
\[ \|\hat K_a\chi_0^2\|_{M_r^q} \leq C\,\|\mathfrak{F}^{-1}(\hat K_a\chi_0^2)\|_{L^p} \leq C_1, \]
with~$C_1>0$, independent of~$a$ and~$1/q=1/r+1/p-1$, for any~$1\leq r<q<\infty$.

%% and Hardy-Littlewood-Sobolev theorem ($\chi_0\in\mathcal C_c^\infty$), it follows that~$\hat K_a\chi_0^2\in M_r^q$ for any~$1<r\leq q<\infty$, and~$\|\hat K_a\|_{M_r^q}$ is uniformly bounded with respect to~$a$.
%
%Let us consider the endpoint estimates, that is~$r=1$ and~$q\in(1,\infty)$. By standard methods, it is easy to prove the pointwise estimates: % (see, for instance, Lemma~8 in~\cite{DAE20})
%%
%\[ |\mathfrak{F}^{-1}(\hat K_a\chi_0^2)(x)|\leq C(1+|x|)^{-n}, \]
%%
%which guarantees, by Young inequality,
%%
%\[ \|\hat K_a\chi_0^2\|_{M_1^q} \leq C\,\|\mathfrak{F}^{-1}(\hat K_a\chi_0^2)\|_{L^q} \leq C_1, \]
%%
%with~$C_1>0$, independent of~$a$.

%On the other hand, due to
%%
%\[ |\hat K_a(\xi) - (1 - a^{2\nu})(2\nu)^{-1} | \leq C\,(a\xii)^2\leq C\xii^2, \]
%%
%we derive by standard methods
%%
%\[ \|\hat K_a\chi_0^2\|_{M_\infty^\infty}=\|\hat K_a\chi_0^2\|_{M_1^1} \leq C_1 + C_2 \|(\hat K_a(\xi) - (1 - a^{2\nu})(2\nu)^{-1})\chi_0^2\|_{M_1^1} \leq C_3. \]
%%
%Therefore, we proved that
%%
%\[ \|\hat K_a\chi_0^2\|_{M_r^q}\leq C,\]
%%
%for any~$1\leq r\leq q\leq\infty$, with~$C>0$, independent of~$a$.

\subsection{$L^{r_1}-L^q$ estimates at intermediate frequencies}

Now $1/2\leq\xii\leq 2a^{-1}$. We may estimate
\[ %J_{\pm\nu}(a\xii)\sim \frac1{\Gamma(1\pm\nu)}\, (a\xii/2)^{\pm\nu},\qquad
|\partial_\xi^\beta J_{\pm\nu}(a\xii)| \lesssim \xii^{-|\beta|}\, (a\xii/2)^{\pm\nu},\]
as we did at low frequencies, but now we use the asymptotic expansion~\eqref{eq:Jasymp} with~$z=\xii$ for~$ J_{\pm\nu}(\xii)$. %
%
%\begin{align*}
%J_{\pm\nu}(\xii)
%    & = (\xii\pi/2)^{-\frac12}\,\cos(\xii\mp\nu\pi/2-\pi/4)\sum_{m=0}^\infty (-1)^m(\nu,2m)(2\xii)^{-2m} \\
%    &\qquad - (\xii\pi/2)^{-\frac12}\,\sin (\xii\mp\nu\pi/2-\pi/4)\sum_{m=0}^\infty (-1)^m(\nu,2m+1)(2\xii)^{-2m-1}.
%\end{align*}
%
We split our analysis in two cases. First, let $1<r_1\leq q$ and $d(r_1,q)\leq \nu+1/2=\mu/2$. Proceeding as we did at low frequencies, by Mikhlin-H\"ormander theorem, recalling that~$\nu>0$, we obtain that
\begin{equation}\label{eq:lowinter}
\|J_{\pm\nu}(a\xii)\,(a\xii)^\nu\,\chi_1\|_{M_q^q}\leq C,
\end{equation}
with~$C$ independent of~$a$, for any~$q\in(1,\infty)$. On the other hand, %thanks to Lemma~\ref{lem:Miy} with~$k=\nu+1/2=\mu/2$ (see also Remark~\ref{rem:Miy}),
\[ \|\xii^{-\nu}\,J_{\pm\nu}(\xii)\,\chi_1\|_{M_{r_1}^q}\leq C, \]
with~$C$ independent of~$a$. % if~$1<r_1\leq q<\infty$, or~$d(r_1,q)<\mu/2$, if~$r_1=1\leq q\leq\infty$ or~$1\leq r_1\leq q=\infty$.
If~$1<r_1\leq q$ and~$d(r_1,q)>\mu/2$, we %shall modify our approach to recover the property~$\hat K_a\chi_1^2\in M_{r_1}^q$, but doing so will lead to a ``loss of decay rate'' represented by a negative power of the parameter~$a$ appearing in the estimate. Indeed, we now
apply Mikhlin-H\"ormander theorem to obtain that
\[ \|J_{\pm\nu}(a\xii)\,(a\xii)^{d(r_1,q)-\frac12}\,\chi_1\|_{M_q^q}\leq C, \]
whereas we %use Lemma~\ref{lem:Miy} with~$k=d(r_1,q)$ (see also Remark~\ref{rem:Miy}) to obtain
estimate
%
%\begin{equation}\label{eq:lowinter2}
\[ a^{\frac\mu2-d(r_1,q)}\,\|\xii^{-d(r_1,q)}\,J_{\pm\nu}(\xii)\,\chi_1\|_{M_{r_1}^q}\leq C\,a^{\frac\mu2-d(r_1,q)}\,.\]
%\end{equation}
%
Summarizing, we proved %that~$\hat K_a\chi_1^2\in M_{r_1}^q$ for any~$1<r_1\leq q<\infty$, with
the following estimate
\[ \|\chi_1^2 \hat K_a \|_{M_{r_1}^q} \leq C \,a^{-(d(r_1,q)-\frac\mu2)_+}\,, \]
where~$C$ is independent of~$a$, for~$1<r_1\leq q<\infty$. If~$r_1=1$ and~$q\in(1,\infty)$, we may proceed as in the first case above if~$d(1,q)<\mu/2$. Otherwise, we proceed as in the second case, but we replace~$d(1,q)$ with~$d(1,q)+\eps$ for some~$\eps>0$. %, using Lemma~\ref{lem:Miy} with~$k=d(1,q)+\eps>d(1,q)$.
\begin{remark}
We notice that the assumption~$\mu>1$, that is, $\nu>0$, is used in~\eqref{eq:lowinter}. For the general case of non integer~$\nu$, we should replace~\eqref{eq:lowinter} by
\begin{equation}\label{eq:lowintersharp}
\|J_{\pm\nu}(a\xii)\,(a\xii)^{|\nu|}\,\chi_1\|_{M_q^q}\leq C.
\end{equation}
%
%Consequently, we would obtain~$\xii^{-|\nu|}\,J_{\pm\nu}(\xii)\,\chi_1 \in M_{r_1}^q$ if, and only if, $d(r_1,q)\leq |\nu|+1/2$.
This modification, eventually, leads to estimate~\eqref{eq:gendecay} without the use of Remark~\ref{rem:change}. %in~\eqref{eq:gendecaysharp}.
\end{remark}

\subsection{$L^{r_2}-L^q$ estimates at high frequencies}

Now~$a\xii\geq1$. In this case, we use the asymptotic expansion~\eqref{eq:Jasymp} for all terms~$J_{\pm\nu}(\xii)$ and~$J_{\pm\nu}(a\xii)$ in~\eqref{eq:Ka} to derive
\[ \hat K_a(\xi) = \frac1{2\sin(\nu\pi)}\,a^{\nu-\frac12} \,\xii^{-1}\,T(a,\xii), \]
with
\begin{align*}
T(a,\xii)
    & = \big(\cos(a\xii+\nu\pi/2-\pi/4)S_0(a\xii)-\sin(a\xii+\nu\pi/2-\pi/4)S_1(a\xii)\big)\\
    & \qquad \times \big(\cos(\xii-\nu\pi/2-\pi/4)S_0(\xii)-\sin(\xii-\nu\pi/2-\pi/4)S_1(\xii)\big) \\
    & \qquad - \big(\cos(a\xii-\nu\pi/2-\pi/4)S_0(a\xii)-\sin(a\xii-\nu\pi/2-\pi/4)S_1(a\xii)\big)\\
    & \qquad \qquad \times \big( \cos(\xii+\nu\pi/2-\pi/4)S_0(\xii)-\sin(\xii+\nu\pi/2-\pi/4)S_1(\xii) \big),
\intertext{where}
S_0(z)
    & =\sum_{m=0}^\infty (-1)^m(\nu,2m)(2z)^{-2m}\,,\qquad S_1(z)=\sum_{m=0}^\infty (-1)^m(\nu,2m+1)(2z)^{-2m-1}\,.
\end{align*}
By addition formulas for the cosine function, we find the leading term%\footnote{We use
%
%\begin{align*}
%& \cos(\alpha+\beta)\cos(\gamma-\beta)-\cos(\alpha-\beta)\cos(\gamma+\beta)\\
%    & = (\cos\alpha\cos\beta-\sin\alpha\sin\beta)(\cos\gamma\cos\beta+\sin\gamma\sin\beta)\\
%    & \qquad - (\cos\alpha\cos\beta+\sin\alpha\sin\beta)(\cos\gamma\cos\beta-\sin\gamma\sin\beta)\\
%    & = 2\cos\alpha\cos\beta\sin\gamma\sin\beta - 2\sin\alpha\sin\beta\cos\gamma\cos\beta\\
%    & = \sin(2\beta)\,(\cos\alpha\sin\gamma - \sin\alpha\cos\gamma) = \sin(2\beta)\,\sin(\gamma-\alpha),
%\end{align*}
%
%with~$\alpha=a\xii-\pi/4$, $\beta=\nu\pi/2$, $\gamma=\xii+\nu\pi/2-\pi/4$.%, as well as
%%
%%\begin{align*}
%%& \sin(\alpha+\beta)\sin(\gamma-\beta)-\sin(\alpha-\beta)\sin(\gamma+\beta)\\
%%    & =(\sin\alpha\cos\beta+\cos\alpha\sin\beta)(\sin\gamma\cos\beta-\cos\gamma\sin\beta)\\
%%    & \qquad - (\sin\alpha\cos\beta-\cos\alpha\sin\beta)(\sin\gamma\cos\beta+\cos\gamma\sin\beta)\\
%%    & = - 2\sin\alpha\cos\gamma\sin\beta\cos\beta + 2 \cos\alpha\sin\gamma\sin\beta\cos\beta\\
%%    & = \sin(2\beta)\,(-\sin\alpha\cos\gamma+\cos\alpha\sin\gamma)=\sin(2\beta)\,\sin(\gamma-\alpha)
%%\end{align*}
%}
%
\[ T(a,\xii)\sim\sin(\nu\pi)\,\sin((1-a)\xii)\,\]
so that
\[ \hat K_a(\xi) = \frac{1-a}{2}\,a^{\nu-\frac12} \,\sinc((1-a)\xii) + \ldots, \]
where we omit the lower order terms in the expansion of~$T(a,\xii)$.

Let~$1<r_2\leq q$, with~$d(r_2,q)\leq1$. Multiplying~$\xii^{-1}\sin((1-a)\xii)$ by~$(a\xii)^{1-d(r_2,q)}$ and applying Lemma~\ref{lem:Miy} with~$k=d(r_2,q)$, we find that
\[ a^{1-d(r_2,q)}\,\|\xii^{-d(r_2,q)}\sin((1-a)\xii)\chi_2^2\|_{M_{r_2}^q} \leq C\,a^{1-d(r_2,q)}, \]
with~$C>0$, independent of~$a$. Due to~$\|(a\xii)^{d(r_2,q)-1}\chi_2^2\|_{M_q^q} \leq C$, as a consequence of Mikhlin-H\"ormander theorem and~$d(r_2,q)\leq1$, we find
\[ \|a^{\nu-\frac12} \,\xii^{-1}\,\sin((1-a)\xii)\chi_2^2\|_{M_{r_2}^q}\leq C\,a^{\frac\mu2-d(r_2,q)}\,, \]
where we replaced~$\nu+1/2=\mu/2$. We proceed similarly with the lower order terms appearing in the expansion of~$T(a,\xii)$, concluding that~$K_a\chi_2^2\in M_{r_2}^q$ if~$d(r_2,q)\leq1$, and
\[ \| \hat K_a\chi_2^2 \|_{M_{r_2}^q}\leq C\,a^{\frac\mu2-d(r_2,q)}\,, \]
with~$C>0$, independent of~$a$. If~$r_2=1$ and~$q\in(1,\infty)$, we may proceed as above if~$d(1,q)<1$, replacing~$d(1,q)$ with~$d(1,q)+\eps$ for any~$\eps\in(0,1-d(1,q)]$, and applying Lemma~\ref{lem:Miy} with~$k=d(1,q)+\eps>d(1,q)$.

\subsection{The case of nonnegative integer values of~$\nu$}\label{sec:integer}

If~$\nu$ is a nonnegative integer, that is, $\mu\in 2\N+1$, then we set
\[ y = C_+(\sigma)\,J_\nu(\tau) + C_-(\sigma)\,\mathbf Y_{\nu}(\tau).\]
where
\[ \mathbf Y_\nu=\lim_{k\to\nu} \frac{J_k-(-1)^\nu J_{-k}}{k-\nu}=\left(\partial_k J_k - (-1)^\nu \partial_k J_{-k}\right)_{k=\nu}, \]
is a Bessel function of second kind. %($\mathbf Y_\nu = \pi\, Y_n$ when~$\nu$ is an integer).
The Wronskian then satisfies~\cite[\textsection 3.63]{Watson} $W[J_\nu,\mathbf Y_\nu](\sigma)=2/\sigma$. Imposing the initial conditions
\[\begin{cases}
C_+ J_\nu(\sigma) + C_- \mathbf Y_{\nu}(\sigma)=0,\\
C_+ J_\nu'(\sigma) + C_- \mathbf Y_{\nu}'(\sigma)=s\sigma^{\nu-1},
\end{cases}\]
we derive
\[ y = -\frac12\,\big( \mathbf Y_{\nu}(\sigma)J_\nu(\tau)-J_\nu(\sigma)\mathbf Y_{\nu}(\tau)\big)\,s\,\sigma^{\nu}, \]
so that, replacing~$\sigma=s\xii$ and~$\tau=t\xii$, we find
\[ \hat K(t,s)= -\frac12\,\big( \mathbf Y_{\nu}(s\xii)J_\nu(t\xii)-J_\nu(s\xii)\mathbf Y_{\nu}(t\xii)\big)\,s^{\nu+1}\,t^{-\nu}. \]
Once again, we study~$\hat K_a$ where
\[ \hat K_a = \hat K(1,a) = -\frac{a^{\nu}}2\,\big( \mathbf Y_{\nu}(a\xii)J_\nu(\xii)- J_\nu(a\xii)\mathbf Y_{\nu}(\xii)\big). \]
The estimates at high frequencies are unchanged, due to the asymptotic expansion (see~\cite[\textsection 7.21]{Watson}):
\begin{align*}
\mathbf Y_{\nu}(z)
    & = (z/(2\pi))^{-\frac12}\,\sin (z-\nu\pi/2-\pi/4)\sum_{m=0}^\infty (-1)^m(\nu,2m)(2z)^{-2m} \\
    &\qquad - (z/(2\pi))^{-\frac12}\,\cos (z-\nu\pi/2-\pi/4)\sum_{m=0}^\infty (-1)^m(\nu,2m+1)(2z)^{-2m-1}.
\end{align*}
However, now
\[ \mathbf Y_0(z)\sim 2\log (z/2),\quad \mathbf Y_\nu(z)\sim -(\nu-1)!\,(z/2)^{-\nu},\quad \nu\in\N\setminus\{0\}, \]
as~$z\to0$, and similarly for their derivatives, using $\mathbf Y_\nu'=\nu z^{-1}\mathbf Y_\nu - \mathbf Y_{\nu+1}$. As a consequence, at low and intermediate frequencies we may still proceed as we did for the case of non-integer~$\nu$ if~$\nu\in\N\setminus\{0\}$. Therefore, let us consider only the case~$\nu=0$, that is, $\mu=1$. In this case, we shall take into account of the logarithmic term.

At low frequencies, for any~$\delta>0$, we easily prove the pointwise estimates (see, for instance, Lemma~8 in~\cite{DAE20})
\[ |\mathfrak{F}^{-1}(\hat K_a\chi_0^2)(x)|\leq C(1+|x|)^{\delta-n}, \]
which guarantees, by Young inequality, $\|\hat K_a\chi_0^2\|_{M_r^q} \leq C$ with~$C>0$, independent of~$a$, for any~$q\in(1,\infty)$ and~$r\in[1,q)$. To recover the estimate for~$r=q$, we notice that cancelations occur in~$\hat K_a$ when~$\nu=0$; indeed,
\[ \hat K_a \sim - \log(a\xii/2) + \log (\xii/2) = -\log a, \]
so that~$\|\hat K_a\chi_0^2\|_{M_q^q} \leq C(1-\log a)$. At intermediate frequencies, we do not have cancelations, but we use that~$(-\log (a\xii))\leq -\log a$, to get
\[ \| \mathbf{Y}_0(a\xii)\,\chi_1 \|_{M_q^q} \leq C(1-\log a),\]
when~$d(r_1,q)\leq1/2$, and
\[ \| \mathbf{Y}_0(a\xii)\,(a\xii)^{d(r_1,q)-1/2}\,\chi_1 \|_{M_q^q} \leq C(1-\log a),\]
when~$d(r_1,q)>1/2$.

\subsection{Unifying the estimates at different frequencies}

Using the estimates obtained for~$\hat K_a$ at low, intermediate and high frequencies, we can now prove Theorem~\ref{thm:main}.

We fix~$r_1,r_2,q$ as in the assumption of Theorem~\ref{thm:main}, and~$\mu>1$. Recalling~\eqref{eq:scaling}, we get
\begin{align*}
\|v(t,\cdot)\|_{L^q}
    & =\|K(t,s)\ast v_1\|_{L^q}\leq \sum_{j=0,1,2}\|\mathfrak{F}^{-1}(\hat K(t,s)\chi_j^2(t\xii))\ast v_1\|_{L^q} \\
    & \leq \sum_{j=0,1}\|\hat K(t,s)\tilde\chi_j^2(t\xii)\|_{M_{r_1}^q}\,\|v_1\|_{L^{r_1}} + \|\hat K(t,s)\tilde\chi_2^2(t\xii)\|_{M_{r_2}^q}\,\|v_1\|_{L^{r_2}}\\
    & = s\,t^{-n\left(\frac1{r_1}-\frac1q\right)}\sum_{j=0,1}\|\hat K_a\chi_j^2\|_{M_{r_1}^q}\,\,\|v_1\|_{L^{r_1}} + s\,t^{-n\left(\frac1{r_2}-\frac1q\right)} \|\hat K_a\chi_2^2\|_{M_{r_2}^q}\,\,\|v_1\|_{L^{r_2}}\\
    & \leq C\,s\,t^{-n\left(\frac1{r_1}-\frac1q\right)}\,(t/s)^{(d(r_1,q)-\frac\mu2)_+}\, \|v_1\|_{L^{r_1}} + C\,s\,t^{-n\left(\frac1{r_2}-\frac1q\right)}\,(t/s)^{d(r_2,q)-\frac\mu2}\,\|v_1\|_{L^{r_2}}\,.
\end{align*}
We replace~$d(r_j,q)$ by~$d(1,q)+\eps$, when~$r_j=1$. This concludes the proof of Theorem~\ref{thm:main} for~$\mu>1$. The case~$\mu=1$ is analogous, taking into account of the logarithmic term considered in Section~\ref{sec:integer}. The proof for~$\mu<1$ follows by using the change of variable in Remark~\ref{rem:change}.

\subsection{Energy estimates}

To deal with energy solutions for the semilinear problems~\eqref{eq:CPsing} and~\eqref{eq:CP}, we supplement the $(L^{r_1}\cap L^{r_2})-L^q$ estimates in Theorem~\ref{thm:main} with the following energy estimates.
\begin{proposition}\label{prop:energy}
Let~$n\geq1$ and~$\mu\in\R$. If~$v_1\in L^1\cap L^2$, then the solution to~\eqref{eq:CPlinear} verifies the following energy estimate:
\[ \|(\nabla v,v_t)(t,\cdot)\|_{L^2} \leq \begin{cases}
C\,s\,t^{-\frac{n}2-1}\big(\|v_1\|_{L^1}+s^{\frac{n}2}\,\|v_1\|_{L^2}\big) & \text{if~$\mu>n+2$},\\
C\,s\,t^{-\frac\mu2}\,(1+\log (t/s))\,\big(\|v_1\|_{L^1}+s^{\frac{n}2}\,\|v_1\|_{L^2}\big) & \text{if~$\mu=n+2$,}\\
C\,t^{-\frac\mu2}\,\big(s^{\frac{\mu-n}2}\,\|v_1\|_{L^1}+s^{\frac\mu2}\,\|v_1\|_{L^2}\big) & \text{if~$-n<\mu<n+2$, $\mu\neq1$,}\\
C\,t^{-\frac12}\,\big(s^{-\frac{n-1}2}\,(1+\log(t/s))\|v_1\|_{L^1}+s^{\frac\mu2}\,\|v_1\|_{L^2}\big) & \text{if~$\mu=1$,}
\end{cases} \]
where~$C>0$ is independent of~$t,s$.
\end{proposition}
For the sake of completeness, we give the straightforward proof of Proposition~\ref{prop:energy}.
\begin{proof}
We follow the proof of Theorem~\ref{thm:main}, but now we consider~$i\xi \hat K(t,s)$ and~$\partial_t\hat K(t,s)$. We notice that
\[ \partial_t\hat K(t,s)=-\dfrac\nu{t}\hat K(t,s) + \xii\,\frac{\pi}{2\sin(\nu\pi)}\, \big( J_{-\nu}(s\xii)J_\nu'(t\xii)-J_\nu(s\xii)J_{-\nu}'(t\xii)\big)\,s^{\nu+1}\,t^{-\nu}, \]
with~$\nu=(\mu-1)/2$ if~$\nu$ is non integer, and similarly using~$\mathbf Y_\nu$ for non-integer values of~$\nu$. Using homogeneity as in~\eqref{eq:scaling}, we now get
\begin{equation}\label{eq:Enscaling}
\|(i\xi,\partial_t)\hat K(t,s)\chi_j^2(t\xii)\|_{M_r^2} = s\,t^{-1-n\left(\frac1r-\frac12\right)}\,\|(i\xi\hat K_a,-\nu\hat K_a + \xii\,m_a)\chi_j^2\|_{M_r^2},
\end{equation}
where we put
\[ m_a=\frac{\pi}{2\sin(\nu\pi)}\, \big( J_{-\nu}(a\xii)J_\nu'(\xii)-J_\nu(a\xii)J_{-\nu}'(\xii)\big)\,a^{\nu}, \]
and similarly for non integer values of~$\nu$. The estimates for~$\|-\nu\,\hat K_a\|_{M_r^2}$ are obtained as in the proof of Theorem~\ref{thm:main}, so in the following we consider~$(i\xi\hat K_a,\xii\,m_a)$.

At low frequencies, it is sufficient to estimate
\[ \|(i\xi\hat K_a,\xii\,m_a)\chi_0^2\|_{M_1^2} \leq \|(i\xi\hat K_a,\xii\,m_a)\chi_0^2\|_{L^2}\leq C\left(\int_{\xii\leq1} \xii^2\,d\xi\right)^{\frac12}=C_1. \]
At intermediate frequencies, if~$\mu>1$, noticing that~$-\nu+1/2=1-\mu/2$, we estimate
\begin{align*}
\|(i\xi\hat K_a,\xii\,m_a)\chi_1^2\|_{M_1^2}
    & \leq C'\,\|\xii^{1-\frac\mu2}\chi_1^2\|_{L^2}\leq C_1\left(\int_{1/2\leq \xii\leq 2a^{-1}} \xii^{2-\mu}\,d\xi\right)^{\frac12} \\
    & \leq \begin{cases}
C & \text{if~$\mu>n+2$},\\
C\,(1-\log a) & \text{if~$\mu=n+2$,}\\
C\,a^{\frac\mu2-\frac{n}2-1} & \text{if~$1<\mu<n+2$.}
\end{cases}
\end{align*}
For~$\mu=1$, that is, $\nu=0$, we get
\[ \|(i\xi\hat K_a,\xii\,m_a)\chi_1^2\|_{M_1^2} \leq C'\,\|\xii^{\frac12}(-\log (a\xii))\chi_1^2\|_{L^2} \leq C\,a^{-\frac{n+1}2}\,(1-\log a). \]
If~$-n<\mu<1$, due to~$-|\nu|+1/2=\mu/2$, we obtain
\[ \|(i\xi\hat K_a,\xii\,m_a)\chi_1^2\|_{M_1^2}
    \leq C'\,a^{2\nu}\, \|\xii^{\frac\mu2}\chi_1^2\|_{L^2} \leq C_1\,a^{2\nu}\, \left(\int_{1/2\leq \xii\leq 2a^{-1}} \xii^{\mu}\,d\xi\right)^{\frac12} \\
    \leq C\,a^{-\frac{n+2-\mu}2}.
\]
Finally, at high frequencies, we simply estimate
\[ \|(i\xi\hat K_a,\xii\,m_a)\chi_2^2\|_{M_2^2} = \|(i\xi\hat K_a,\xii\,m_a)\chi_2^2\|_{L^\infty} \leq C\,a^{\frac{\mu}2-1}. \]
Unifying the estimates and using~\eqref{eq:Enscaling}, we derive
\begin{align*}
\|(\nabla,\partial_t)v(t,\cdot)\|_{L^2}
    & \leq C\,t^{-\frac\mu2}\,s^{\frac\mu2}\,\|v_1\|_{L^2} \\
    & \qquad + \|v_1\|_{L^1} \times \begin{cases}
C\,s\,t^{-\frac{n}2-1} & \text{if~$\mu>n+2$},\\
C\,s\,t^{-\frac{n}2-1}\,(1+\log (t/s)) & \text{if~$\mu=n+2$,}\\
C\,s^{\frac{\mu-n}2}\,t^{-\frac\mu2} & \text{if~$-n<\mu<n+2$, $\mu\neq1$,}\\
C\,s^{-\frac{n-1}2}\,t^{-\frac12}\,(1+\log(t/s)) & \text{if~$\mu=1$,}
\end{cases}
\end{align*}
and this concludes the proof.
\end{proof}

\subsection{The singular problem for the Euler-Poisson-Darboux equation}\label{sec:sing}

We consider the singular linear Cauchy problem
\begin{equation}\label{eq:CPsinglin}
\begin{cases}
v_{tt}-\triangle v+ \dfrac\mu{t}\, v_t=0, & t>0, \ x\in\R^n\,,\\
v(0,x)=v_0(x)\,, \quad v_t(0,x)=0\,.
\end{cases}
\end{equation}
We stress that for this singular problem the assumption~$v_t(0,x)=0$ is natural, so we assume that the initial datum~$v(0,x)$ is not identically zero, to exclude the null solution.

For~$\mu>0$, it is easy to check that the solution to the linear singular problem~\eqref{eq:CPsinglin} verifies (see also~\cite{Bre}):
\[ \hat v(t,\xi) = 2^\nu\,\Gamma(1+\nu)\,(t\xii)^{-\nu}\,J_\nu(t\xii)\hat v_0(\xi),\]
where~$\nu=(\mu-1)/2$. As a consequence, following as in the proof of Theorem~\ref{thm:main}, for~$q\in(1,\infty)$ and~$r\in[1,q]$, it is easy to prove the $L^r-L^q$ estimate
\begin{equation}\label{eq:rqsing}
\|v(t,\cdot)\|_{L^q} \leq C\,t^{-n\left(\frac1r-\frac1q\right)}\,\|v_0\|_{L^r},
\end{equation}
provided that~$d(r,q)\leq \mu/2$, or that the strict inequality holds if~$r=1$. %FALSO: The previous condition may be violated if we replace~$\|u_1\|_{L^r}$ by~$\|u_1\|_{W^{k,r}}$ with~$k\geq d(r,q)-\mu/2$ if~$r>1$, or~$k>d(1,q)-\mu/2$, even integer, if~$r=1$.

%When~$\mu<1$, we rely on Remark~\ref{rem:change}. In turn, this gives %%%% FALSO PERCHE' NON PUOI FARE IL CAMBIO
%%
%
%\|v(t,\cdot)\|_{L^q} \leq C\,t^{(1-\mu)_+-n\left(\frac1r-\frac1q\right)}\,\|v_0\|_{L^r},
%\end{equation}
%%
%provided that~$d(r,q)\leq \max\{\mu,2-\mu\}/2$, or that the strict inequality holds if~$r=1$.

%For any~$q\in(1,\infty)$ with~$d(q,q)>\mu/2$, the $L^q-L^q$ estimate may be easily replaced by the~$W^{\kappa,q}-L^q$ estimate: NON CI SERVE
%%
%\[ \|v(t,\cdot)\|_{L^q} \leq C\,(1+t)^\kappa\,\|v_0\|_{W^{\kappa,q}}, \qquad \kappa = d(q,q)-\frac\mu2 = (n-1)\left|\frac1q-\frac12\right|-\frac\mu2, \]
%%
%where~$W^{\kappa,q}$ is the Bessel potential space~$W^{\kappa,q}=\{f\in L^q: \ \mathfrak{F}^{-1}(\xii^\kappa \hat f)\in L^q\}$.
%We stress that if~$\mu\geq1$, the decay rate does not depend on~$\mu$, whereas it does for the nonsingular problem~\eqref{eq:CPlinear}.

Similarly, the energy estimates
\begin{equation}\label{eq:energysing}
\|(\partial_t,\nabla)v(t,\cdot)\|_{L^2} \leq \begin{cases}
t^{-\frac{n}2-1}\,\|v_0\|_{L^1} & \text{if~$\mu>n+2$,}\\
t^{-n\left(\frac1r-\frac12\right)-1}\,\|v_0\|_{L^r} & \text{if~$r\in(1,2]$ and~$\mu\geq n(2/r-1)+2$,}\\
t^{-\kappa}\, \|v_0\|_{H^{1-\kappa}} & \text{if~$\kappa\in[0,1)$ and~$\mu\geq2\kappa$,}
\end{cases}
\end{equation}
may be easily derived for the solution to~\eqref{eq:CPsinglin}, following the proof of Proposition~\ref{prop:energy}.

%%%%%%%%%%%%%%%%%%%%%%%%%%%%%%%%%%%%%%%%%%%%%%%%%%%%%%%%%%%%%

\section{Proof of Theorem~\ref{thm:main1}}\label{sec:proof1}

We first prove a generalization of Theorem~\ref{thm:main1} for the regular Cauchy problem, in which we replace the nonlinearity~$f(u)$ by~$t^{-\alpha}\,f(u)$, for~$\alpha\in[0,2)$, namely, we consider
\begin{equation}\label{eq:CPalpha}
\begin{cases}
u_{tt}-\Delta u + \dfrac\mu{t}\, u_t= t^{-\alpha}\,f(u), & t\geq t_0>0, \ x\in\R^n\,,\\
u(t_0,x)=0\,, \quad u_t(t_0,x)=u_1(x)\,,
\end{cases}
\end{equation}
instead of~\eqref{eq:CP}, in space dimension~$n=1,2$. For this model, we set
\begin{equation}\label{eq:pcritalpha}
p_\crit = \max \left\{ 1 + \frac{2-\alpha}{n-1+\min\{1,\mu\}}, \quad p_\Str(n+\mu,\alpha)\right\},
\end{equation}
where~$1+(2-\alpha)/(n-1+\min\{1,\mu\})$ is a \emph{modified Fujita exponent}, and~$p_\Str(n+\mu,\alpha)$ is a \emph{modified, shifted Strauss exponent}: for a given~$k>1$, we set~$p_\Str(k,\alpha)$ as the solution to
\begin{equation}\label{eq:Straussmod}
\frac{k-1}2(p-1)-(1-\alpha)-\frac1{p}=0.
\end{equation}
The exponent in~\eqref{eq:Straussmod} also appears later on, see Remark~\ref{rem:Tricomi}.

We remark that, in particular, $p_\crit\searrow1$ as~$\alpha\nearrow2$, for any~$\mu>1-n$.
\begin{remark}\label{rem:1p}
The critical exponent in~\eqref{eq:pcritalpha} is related to the $L^1-L^{p}$ decay rate in Corollary~\ref{cor:main}. Indeed, the value~$p_\crit$ is the power~$p$ such that the exponent of~$t$, multiplied by~$p$, summed to the exponent of~$s$, gives~$\alpha-1$, if we set~$r=1$ and~$q=p$ in \eqref{eq:rqdecaygood} or~\eqref{eq:rqdecaybad}.
\end{remark}
%
%To prove our result, we will use Corollary~\ref{cor:main}, setting~$r=1$. The key estimate for the nonlinear argument will be the one related to the $L^1-L^{p_\crit}$ estimate.
%
\begin{remark}\label{rem:optimal}
The exponent~$1+(2-\alpha)/(n-1+\min\{1,\mu\})$ in~\eqref{eq:pcritalpha} is a modified Fujita exponent and no global-in-time solutions exist if~$p\in(1,p_\crit]$, for suitable sign assumption on the initial datum. If~$\mu\geq1$, this follows by applying Theorem 1.1 in~\cite{DAL13}. If~$\mu<1$, it also follows by the same theorem, after using the change of variable $u^\sharp=t^{\mu-1}u$, as in Remark~\ref{rem:change}. Namely, \eqref{eq:CPalpha} is equivalent to
\begin{equation}\label{eq:CPalphasharp}
\begin{cases}
u_{tt}^\sharp- \Delta u^\sharp + \dfrac{2-\mu}{t}\,u_t^\sharp = t^{-\alpha-\frac\mu2(p-1)}\,f(u), & t>s\,, \ x\in\R^n\,,\\
u^\sharp(s,x)=0, \quad u_t^\sharp(s,x)=s^{1-\mu}\,u_1(x).
\end{cases}
\end{equation}
We also expect that the blow-up result in~\cite{IkedaSob17} may be extended to more general values of~$\alpha>0$, leading to a nonexistence result for~$p\leq p_\Str(n+\mu,\alpha)$, where~$p_\Str(k,\alpha)$ is defined by~\eqref{eq:pcritalpha}.
\end{remark}
We are now ready to state our main result for~\eqref{eq:CPalpha} in space dimension~$n=1$ (for a result in space dimension~$n=2$ see later, Section~\ref{sec:weakn2}).
\begin{theorem}\label{thm:maingen}
Fix~$n=1$, $\alpha\in[0,2)$, $\mu>0$ and~$p>p_\crit$, where~$p_\crit$ is as in~\eqref{eq:pcritalpha}. Then there exists~$\eps>0$ such that for any initial data as in~\eqref{eq:data}, there exists a unique $u\in\mathcal C([t_0,\infty),H^1)\cap\mathcal C^1([t_0,\infty),L^2)$, global-in-time energy solution to~\eqref{eq:CPalpha}. Moreover, the energy estimate
\begin{align}\nonumber
E(t)
    & = \frac12\,\|u_t(t,\cdot)\|_{L^2}^2 + \frac12\,\|u_x(t,\cdot)\|_{L^2}^2\\
    \label{eq:energyest}
    & \leq C\,\big(\|u_1\|_{L^1}^2+\|u_1\|_{L^2}^2\big)\times\begin{cases}
t^{-3} & \text{if~$\mu>3$},\\
t^{-3}\,(1+\log (t/t_0)) & \text{if~$\mu=3$,}\\
t^{-\mu} & \text{if~$0<\mu<3$ with~$\mu\neq1$,}\\
t^{-1}(1+\log(t/t_0))^2 & \text{if~$\mu=1$,}
\end{cases}
\end{align}
holds, $u\in L^\infty([t_0,\infty),L^{p_\crit})$ if~$p_\crit<2$, and for any~$\delta>0$, we have the decay estimate
\begin{equation}\label{eq:decay}
\|u(t,\cdot)\|_{L^q} \leq C\,\big(\|u_1\|_{L^1}+\|u_1\|_{L^2}\big)\times \begin{cases}
t^{(1-\mu)_+-1+\frac1q} & \text{if~$2-2/q<\max\{\mu,2-\mu\}$,} \\
t^{\delta-\frac\mu2} & \text{if~$2-2/q\geq \max\{\mu,2-\mu\}$,}
\end{cases}
\end{equation}
where~$C=C(t_0)>0$, for any~$q\in[p_\crit,\infty)$.
\end{theorem}
\begin{remark}\label{rem:findp}
Let us determine~$p_\crit$ according to the value of~$\mu$, in space dimension~$n=1$. We stress that
\[ 1 + \frac{2-\alpha}{\min\{1,\mu\}}>p_\Str(1+\mu,\alpha) \iff 2-\max\{\mu,2-\mu\}<\frac2{p_\crit}\,. \]
It holds~$p_\crit=p_\Str(1+\mu,\alpha)$ if, and only if, $\alpha\in[0,1]$ and~$\alpha\leq\mu\leq\bar\mu$, where
\begin{equation}\label{eq:barmualpha}
\bar \mu = \frac{2(2-\alpha)}{3-\alpha}.
\end{equation}
It holds~$p_\crit=3-\alpha$ if, and only if, either $\mu\geq\bar\mu$, when~$\alpha\in[0,1)$, or~$\mu\geq1$ when~$\alpha\in[1,2)$. It holds~$p_\crit=1+(2-\alpha)/\mu$ if, and only if, $0<\mu\leq\alpha$ if~$\alpha\in(0,1)$, or~$\mu\leq1$ if~$\alpha\in[1,2)$. In particular, we stress that this case never occurs when~$\alpha=0$.
\end{remark}
Setting~$\alpha=0$ in Theorem~\ref{thm:maingen}, we have a result for~\eqref{eq:CP}, but for~$\alpha=0$ we may remove the dependence of~$C$ from~$t_0$, replacing~$t$ by~$1+t$ in estimates~\eqref{eq:energyest} and~\eqref{eq:decay}.

With minor modifications, we may state the analogous of Theorem~\ref{thm:maingen} for the singular Cauchy problem with nonlinearity~$t^{-\alpha}f(u)$, but some technical restrictions on~$\alpha$ with respect to~$\mu$ appear, related to the singularity of~$s^{-\alpha}$ as~$s\to0$ in the Duhamel's integral. These restrictions may be relaxed if weak solutions are looked for, instead of energy solutions. We restricted Theorem~\ref{thm:main1} to~$\alpha=0$ for the ease of reading, being the generalization to~$\alpha>0$ quite straightforward.

\subsection{Proof of Theorem~\ref{thm:maingen}}

To prove Theorem~\ref{thm:maingen}, we use a standard contraction argument, exploiting the sharpness of the $L^r-L^q$ decay estimates derived in Corollary~\ref{cor:main}, to construct a suitable solution space, in which we may prove the global-in-time existence of small data solutions to~\eqref{eq:CPalpha} for~$p>p_\crit$.
\begin{proof}[Proof of Theorem~\ref{thm:maingen}]
For a general~$T>t_0$, we fix~$X(T)$ as a subspace of the energy space~$\mathcal C([t_0,T],H^1)\cap\mathcal C^1([t_0,T],L^2)$ if~$p_\crit\geq2$, or as a subspace of the space $\mathcal C([t_0,T],H^1)\cap\mathcal C^1([t_0,T],L^2)\cap L^\infty([t_0,T],L^{p_\crit})$ if~$p_\crit<2$. We prove that there exists a constant~$C=C(t_0)$, independent of~$T$, such that:
\begin{itemize}
\item the solution to the linear problem~\eqref{eq:CPlinear} with~$s=t_0$ and~$v_1=u_1$ verifies the estimate
\begin{equation}
\label{eq:dataX}
\|v\|_{X(T)} \leq C\,\big(\|u_1\|_{L^1}+\|u_1\|_{L^2}\big)\,;
\end{equation}
\item the operator
\[ F: X(T) \to X(T), \quad Fu(t,x) = \int_{t_0}^t K(t,s)\ast f(u(s,x))\,ds, \]
where~$K(t,s)$ is the fundamental solution to~\eqref{eq:CPlinear}, verifies the contractive estimate
\begin{equation}
\label{eq:contraX}
\|Fu-Fw\|_{X(T)} \leq C\,\|u-w\|_{X(T)}\big(\|u\|_{X(T)}^{p-1}+\|w\|_{X(T)}^{p-1}\big).
\end{equation}
\end{itemize}
By standard contraction arguments, properties \eqref{eq:dataX}-\eqref{eq:contraX} imply that there exists~$\eps>0$ such that if~$u_1$ verifies~\eqref{eq:data}, then there is a unique global-in-time solution to~\eqref{eq:CPalpha}, verifying
\[ \|u\|_{X(T)} \leq C\,\big(\|u_1\|_{L^1}+\|u_1\|_{L^2}\big)\,, \]
for any~$T>t_0$, with~$C=C(t_0)$, independent of~$T$. Indeed, let~$R>0$ be such that~$CR^{p-1}<1/2$. Then~$F$ is a contraction on~$X_R(T)= \{u\in X(T): \ \|u\|_{X(T)}\leq R\}$. The solution to~\eqref{eq:CPalpha} is a fixed point for~$v(t,x)+Fu(t,x)$, so if~$\|v\|_{X(T)}\leq R/2$, then~$u\in X_R(T)$ and the uniqueness and existence of the solution in~$X_R(T)$ follows by the Banach fixed point theorem on contractions. The condition~$\|v\|_{X(T)}\leq R/2$ is obtained taking initial data as in~\eqref{eq:data}, with~$C\eps\leq R/2$. Since~$C$, $R$ and~$\eps$ do not depend on~$T$, the solution is global-in-time.

The continuity of the maps~$t\mapsto u(t,\cdot)\in H^1$ and~$t\mapsto u_t(t,\cdot)\in L^2$ is standard and so we will omit its proof. Then, in the following, we prove~\eqref{eq:dataX} and~\eqref{eq:contraX} for a suitable choice of~$X(T)$ and its norm.

We fix the notation
\begin{equation}\label{eq:eng}
g(t)= \begin{cases}
t^{-\frac32} & \text{if~$\mu>3$,}\\
t^{-\frac\mu2}\,(1+\log (t/t_0)) & \text{if~$\mu=3$,}\\
t^{-\frac\mu2} & \text{if~$0<\mu<3$, $\mu\neq1$,}\\
t^{-\frac12}\,(1+\log(t/t_0)) & \text{if~$\mu=1$,}
\end{cases}
\end{equation}
to describe the decay rate for the energy estimates, according to the different values of~$\mu$, and for any~$q\in[p_\crit,\infty)$, we set
\begin{equation}\label{eq:gammaq}
\gamma_q=\begin{cases}
1-\frac1q - (1-\mu)_+ & \text{if~$2-2/q<\max\{\mu,2-\mu\}$,} \\
\frac\mu2-\delta & \text{if~$2-2/q\geq\max\{\mu,2-\mu\}$,}
\end{cases}
\end{equation}
for a sufficiently small~$\delta>0$, which we will fix later.

Let~$X(T)$ be the subspace of functions in~$\mathcal C([t_0,T],H^1)\cap\mathcal C^1([t_0,T],L^2)$, verifying~$\|u\|_{X(T)}<\infty$, where
\begin{equation}\label{eq:normaX}
\|u\|_{X(T)} = \sup_{t\in[t_0,T]} \Bigl( (g(t))^{-1}\, \|(u_t,u_x)(t,\cdot)\|_{L^2} + \sup \big\{ t^{\gamma_q}\,\|u(t,\cdot)\|_{L^q}: \ q\in[p_\crit,\infty)\big\} \Bigr).
\end{equation}
With this choice of norm on~$X(T)$, the solution to the linear problem~\eqref{eq:CPlinear} verifies~\eqref{eq:dataX}.

Indeed, if~$2-2/q<\max\{\mu,2-\mu\}$, by~\eqref{eq:rqdecaygood} with~$r=1$, we obtain
\begin{equation}\label{eq:1qdecaygood}
\|v(t,\cdot)\|_{L^q} \leq C\,s^{\min\{1,\mu\}}\,t^{(1-\mu)_+-1+\frac1q}\, \|v_1\|_{L^1} \,,
\end{equation}
for some~$C>0$, independent of~$s,t$. If~$2-2/q\geq\max\{\mu,2-\mu\}$, by~\eqref{eq:rqdecaybad} with~$r=1$, we obtain
\begin{equation}\label{eq:1qdecaybad}
\|v(t,\cdot)\|_{L^q} \leq C\,s^{\frac1q+\frac\mu2-\delta}\,t^{\delta-\frac\mu2}\, \|v_1\|_{L^1}\,,
\end{equation}
taking~$\eps\leq\delta$. Setting~$s=t_0$, thanks to \eqref{eq:1qdecaygood}-\eqref{eq:1qdecaybad}, and thanks to Proposition~\ref{prop:energy}, we get~\eqref{eq:dataX}.

Now let~$u,w\in X(T)$. We want to prove~\eqref{eq:contraX}.

First, we assume that~$p_\crit=1+(2-\alpha)/\min\{1,\mu\}$, that is, $2-2/p_\crit<\max\{\mu,2-\mu\}$ (see Remark~\ref{rem:findp}).

Let~$q\in[p_\crit,\infty)$. We consider two cases.

If~$q$ verifies $2-2/q<\max\{\mu,2-\mu\}$, using~\eqref{eq:1qdecaygood}, we get
\begin{align*}
& t^{-(1-\mu)_++1-\frac1q} \, \|(Fu-Fw)(t,\cdot)\|_{L^q}\\
    & \qquad \leq C\, \int_{t_0}^t s^{\min\{1,\mu\}-\alpha}\,\| (f(u)-f(w))(s,\cdot) \|_{L^1}\,ds \\
    & \qquad \leq C\,\int_{t_0}^t s^{\min\{1,\mu\}-\alpha-p\gamma_p}\,\,ds\,\|u-w\|_{X(T)}\big( \|u\|_{X(T)}^{p-1}+\|w\|_{X(T)}^{p-1}\big),
\end{align*}
where we used~\eqref{eq:fu} with H\"older inequality, and the fact that~$u,w\in X(T)$, to estimate
\begin{equation}\label{eq:fu1}
\begin{split}
\|(f(u)-f(w))(s,\cdot)\|_{L^1}
    & \leq C\, \|(u-w)(s,\cdot)\|_{L^p}\,\big( \|u(s,\cdot)\|_{L^p}^{p-1}+\|w(s,\cdot)\|_{L^p}^{p-1}\big) \\
    & \leq C\,s^{-p\gamma_p}\,\|u-w\|_{X(T)}\big( \|u\|_{X(T)}^{p-1}+\|w\|_{X(T)}^{p-1}\big).
\end{split}\end{equation}
We now want to prove that the integral
\[ \int_{t_0}^t s^{\min\{1,\mu\}-\alpha-p\gamma_p}\,\,ds \]
is uniformly bounded, with respect to~$t$, that is, that the integral
\[ \int_{t_0}^\infty s^{\min\{1,\mu\}-\alpha-p\gamma_p}\,\,ds \]
is convergent. This latter is true if, and only if, $p\gamma_p+(1-\mu)_+>2-\alpha$. This condition holds for any~$p>p_\crit=1+(2-\alpha)/\min\{1,\mu\}$, due to
\[ p\gamma_p +(1-\mu)_+> p_\crit\gamma_{p_\crit} +(1-\mu)_+ = (1-(1-\mu)_+) (p_\crit-1)=\min\{1,\mu\}(p_\crit-1).\]
Now let~$q$ verify~$2-2/q\geq\max\{\mu,2-\mu\}$. We now use~\eqref{eq:1qdecaybad} to get
\begin{equation}\label{eq:bad1}
\begin{split}
t^{\frac\mu2-\delta} \, \|(Fu-Fw)(t,\cdot)\|_{L^q}
    & \leq C\, \int_{t_0}^t s^{-\alpha+\frac1q+\frac\mu2-\delta}\,\| f(u(s,\cdot))-f(w(s,\cdot)) \|_{L^1}\,ds \\
    & \leq C\,\int_{t_0}^t s^{-\alpha+\frac1q+\frac\mu2-\delta-p\gamma_p}\,\,ds\,\|u-w\|_{X(T)}\big( \|u\|_{X(T)}^{p-1}+\|w\|_{X(T)}^{p-1}\big).
\end{split}\end{equation}
This integral is also uniformly bounded, with respect to~$t$, due to
\[ \frac1q+\frac\mu2-\delta \leq 1-\frac{\max\{\mu,2-\mu\}}2+\frac\mu2-\delta = \min\{1,\mu\}-\delta <\min\{1,\mu\}. \]
Thus, we proved
\[ t^{\gamma_q} \,\|Fu(t,\cdot)-Fw(t,\cdot)\|_{L^q} \leq C(t_0)\,\|u-w\|_{X(T)}\big( \|u\|_{X(T)}^{p-1}+\|w\|_{X(T)}^{p-1}\big)\,,\]
for any~$q\in[p_\crit,\infty)$.

We now consider the energy estimates. Exception given for the cases~$\mu=3,1$, for which an additional logarithmic power of~$s$ appears in the integral, by Proposition~\ref{prop:energy}, we obtain:
\begin{align*}
& (g(t))^{-1}\,\|(\partial_t,\partial_x)(Fu-Fw)(t,\cdot)\|_{L^2}\\
    & \qquad \leq C\,\int_{t_0}^t s^{-\alpha+\min\{1,\frac{\mu-1}2\}}\,\big(\| (f(u)-f(w))(s,\cdot) \|_{L^1}+s^{\frac12}\,\| (f(u)-f(w))(s,\cdot) \|_{L^2}\big)\,ds \\
    & \qquad \leq C\,\int_{t_0}^t \big(s^{-\alpha+\frac{\min\{2,\mu-1\}}2-p\gamma_p}+s^{-\alpha+\frac{\min\{3,\mu\}}2-p\gamma_{2p}}\big)\,\,ds\,\|u-w\|_{X(T)}\big( \|u\|_{X(T)}^{p-1}+\|w\|_{X(T)}^{p-1}\big)\\
    & \qquad \leq C(t_0)\,\|u-w\|_{X(T)}\big( \|u\|_{X(T)}^{p-1}+\|w\|_{X(T)}^{p-1}\big).
\end{align*}
Here we used
\[ -\alpha+\frac{\min\{2,\mu-1\}}2-p\gamma_p \leq -\alpha+1-p\gamma_p<-\alpha+1-p_\crit\gamma_{p_\crit}=-1, \]
if~$\mu>1$, or
\[ -\alpha+\frac{\min\{2,\mu-1\}}2-p\gamma_p < -\alpha-p_\crit\gamma_{p_\crit}<-1, \]
if~$\mu\in(0,1)$. Similarly, if~$2-1/p_\crit<\max\{\mu,2-\mu\}$, then
\[ -\alpha+\frac{\min\{3,\mu\}}2-p\gamma_{2p_\crit} < -\alpha+1+\frac12 - (p_\crit-1/2) =-1  \]
if~$\mu>1$, or
\[ -\alpha+\frac{\min\{3,\mu\}}2-p\gamma_{2p_\crit} < -\alpha+\frac\mu2 - \mu(p_\crit-1/2)=-2, \]
if~$\mu\in(0,1)$. On the other hand, if~$2-1/p_\crit\geq\max\{\mu,2-\mu\}$ (so that~$\mu\in(0,2)$), then
\[
-\alpha+\frac{\min\{3,\mu\}}2-p\gamma_{2p_\crit} = -\alpha+\frac\mu2 -p\left( \frac\mu2-\delta\right) < -\alpha-\frac{\mu}2\,(p_\crit-1)\leq-1.
\]
Summarizing, we proved~\eqref{eq:contraX}, and this concludes the proof of Theorem~\ref{thm:maingen} when~$p_\crit=1+(2-\alpha)/\min\{1,\mu\}$.

\bigskip

We now assume that~$p_\crit=p_\Str(1+\mu,\alpha)$, that is, $2-2/p_\crit\geq \max\{\mu,2-\mu\}$ (see Remark~\ref{rem:findp}). In this case, $\gamma_q=\mu/2-\delta$ for any~$q\in[p_\crit,\infty)$. %Let~$q\in[p_\crit,\infty)$.
Using~\eqref{eq:1qdecaybad}, we find again~\eqref{eq:bad1}, and we get that the integral is uniformly bounded with respect to~$t$, if, and only if,
\[ -\alpha+\frac1q+\frac\mu2-\delta-p\left(\frac\mu2 -\delta \right) <-1 \]
for some~$\delta>0$. The above condition is verified for any~$p>\bar p$ and~$q\geq\bar p$, for a sufficiently small~$\delta$, depending on~$p$, if, and only if,
\[ \frac{\mu}2\,(\bar p-1) - (1-\alpha) -\frac1{\bar p}\geq0, \]
that is, $\bar p\geq p_\Str(1+\mu,\alpha)$. Thus, we proved
\[ t^{\frac\mu2-\delta}\,\|Fu(t,\cdot)-Fw(t,\cdot)\|_{L^q} \leq C(t_0)\,\|u-w\|_{X(T)}\big( \|u\|_{X(T)}^{p-1}+\|w\|_{X(T)}^{p-1}\big)\,,\]
for any~$q\in[p_\crit,\infty)$.

For the energy estimates, by Proposition~\ref{prop:energy}, exception given for the case~$\mu=1$ for which a logarithmic power appears, we easily obtain:
\begin{align*}
& t^{\frac\mu2}\|(\partial_t,\partial_x)(Fu-Fw)(t,\cdot)\|_{L^2}\\
    & \qquad \leq C\, \int_{t_0}^t s^{-\alpha+\frac{\mu-1}2}\,\big(\| (f(u)-f(w))(s,\cdot) \|_{L^1}+s^{\frac12}\,\| (f(u)-f(w))(s,\cdot) \|_{L^2}\big)\,ds \\
    & \qquad \leq C\, \int_{t_0}^t s^{-\alpha+\frac{\mu-1}2-p\left(\frac\mu2-\delta\right)}\,\big(1+s^{\frac12} \big)\,ds\,\|u-w\|_{X(T)}\big( \|u\|_{X(T)}^{p-1}+\|w\|_{X(T)}^{p-1}\big)\\
    & \qquad \leq C(t_0)\,\|u-w\|_{X(T)}\big( \|u\|_{X(T)}^{p-1}+\|w\|_{X(T)}^{p-1}\big),
\end{align*}
where the integral is uniformly bounded with respect to~$t$, for a sufficiently small~$\delta$, depending on~$p$, due to
\[ -\alpha+\frac{\mu}2-p\left(\frac\mu2-\delta\right) < -\frac{\mu}2\,(p_\crit-1)-\alpha =-1-\frac1{p_\crit} <-1,\]
where we used~$p>p_\crit =p_\Str(1+\mu,\alpha)$. This concludes the proof of Theorem~\ref{thm:maingen}.
\end{proof}

%%%%%%%%%%%%%%%%%%%%%%%%%%%%%%%%%%%%%%%%%%%%%%%%%%%%%%%%%%%%%%%%%%%

\subsection{Proof of Theorem~\ref{thm:main1}}

Setting~$\alpha=0$, we may conveniently modify the proof of Theorem~\ref{thm:maingen} to prove Theorem~\ref{thm:main1}.
\begin{proof}[Proof of Theorem~\ref{thm:main1}]
We follow the proof of Theorem~\ref{thm:maingen}, with appropriate modifications to manage the singularity at~$t=0$ in the case of the singular problem~\eqref{eq:CPsing}, and to remove the dependence on~$t_0$ from the constants in the case of the regular problem~\eqref{eq:CP}.

We now want to prove that for a given~$T>0$, there exists a constant~$C>0$, independent of~$T$, such that either the solution to the linear problem~\eqref{eq:CPsinglin} with~$v_0=u_0$ verifies the estimate
\begin{equation}
\label{eq:dataXsing}
\|v\|_{X(T)} \leq C\,\big(\|u_0\|_{L^1}+\|u_0\|_{H^1}\big)\,,
\end{equation}
or the solution to the linear problem~\eqref{eq:CP} with~$v_1=u_1$ verifies~\eqref{eq:dataX}, and \eqref{eq:contraX} holds, with~$C$ independent of~$T$ and~$t_0$. For the ease of notation, in the following, we fix~$t_0=0$ if we consider the singular problem~\eqref{eq:CPsing}.

By standard contraction arguments (as in the proof of Theorem~\ref{thm:maingen}), properties~\eqref{eq:dataXsing} and~\eqref{eq:contraX} imply that there exists~$\eps>0$ such that if~$u_0$ verifies~\eqref{eq:datasing}, then there is a unique global-in-time solution to~\eqref{eq:CPsing}, verifying
\[ \|u\|_{X(T)} \leq C\,\big(\|u_0\|_{L^1}+\|u_0\|_{H^1}\big)\,, \]
for any~$T>0$, with~$C>0$, independent of~$T$, and similarly for problem~\eqref{eq:CP}. We now define~$g$ by
\begin{equation}\label{eq:engsing}
g(t)= \begin{cases}
(1+t)^{-\frac32} & \text{if~$\mu>3$,}\\
(1+t)^{\delta-\frac\mu2} & \text{if~$\mu=3$,}\\
(1+t)^{-\frac\mu2} & \text{if~$0<\mu<3$, $\mu\neq1$,}\\
(1+t)^{-\frac12}\,(1+\log(1+t)) & \text{if~$\mu=1$,}
\end{cases}
\end{equation}
whereas~$\gamma_q$ is as in~\eqref{eq:gammaq}. Now~$X(T)$ is the subspace of functions in~$\mathcal C([t_0,T],H^1)\cap\mathcal C^1([t_0,T],L^2)$, verifying~$\|u\|_{X(T)}<\infty$, where
\begin{equation}\label{eq:normaXsing}
\|u\|_{X(T)} = \sup_{[t_0,T]} \Bigl( (g(t))^{-1}\, \|(u_t,u_x)(t,\cdot)\|_{L^2} + \sup \big\{ (1+t)^{\gamma_q}\,\|u(t,\cdot)\|_{L^q}: \ q\in[p_\crit,\infty)\big\} \Bigr).
\end{equation}
Let us prove that the solution to the linear singular problem~\eqref{eq:CPsinglin} verifies~\eqref{eq:dataXsing}. Fix~$q\in[p_\crit,\infty)$. We distinguish estimates at short time~$t\in[0,1]$ and at a long time~$t\in[1,T]$. At short time~$t\leq1$, it is sufficient to employ~\eqref{eq:rqsing} with~$r=q$ and get $\|v(t,\cdot)\|_{L^q}\leq C\,\|u_0\|_{L^q}$. At long time~$t\geq1$, we may use~\eqref{eq:rqsing} with~$r=1$ if~$1-1/q<\mu/2$, to get
\[ \|v(t,\cdot)\|_{L^q}\leq C\,t^{-1+\frac1q}\,\|u_0\|_{L^1}\,. \]
If~$1-1/q>\mu/2$, we fix~$r\in(1,q)$ such that~$1-1/q=\mu/2$, so that by~\eqref{eq:rqsing}, we obtain
\[ \|v(t,\cdot)\|_{L^q}\leq C\,t^{-\frac\mu2}\,\|u_0\|_{L^r}. \]
In the limit case~$1-1/q=\mu/2$, we fix~$r\in(1,q)$ such that~$1-1/q=\mu/2-\delta$ so that by~\eqref{eq:rqsing} by get
\[ \|v(t,\cdot)\|_{L^q}\leq C\,t^{\delta-\frac\mu2}\,\|u_0\|_{L^r}. \]
We notice that when~$\mu\in(0,1)$ and~$\mu/2\leq 1-1/q<1-\mu/2$, we may still estimate
\[ -\frac\mu2 < (1-\mu)_+ -1+\frac1q = \gamma_q. \]
We proceed similarly for the energy estimates. At short time~$t\leq1$, it is sufficient to use~\eqref{eq:energysing} with~$\kappa=0$, whereas at long time~$t\geq1$ we use~\eqref{eq:energysing} with~$r=1$ if~$\mu>3$, with~$r\in(1,2]$ such that~$\mu=2/r+1$ if~$\mu\in[2,3)$, and such that~$1/r-1/2=\delta$ if~$\mu=3$, and with~$\kappa=\mu/2$ if~$\mu\in(0,2)$. This proves~\eqref{eq:dataXsing}.

Similarly, we prove~\eqref{eq:dataX} with~$C$ independent of~$T,t_0$.

Now let~$u,w\in X(T)$. We want to prove~\eqref{eq:contraX}. We proceed as in the proof of Theorem~\ref{thm:maingen}, but we notice that the definition of the norm in~\eqref{eq:normaXsing} means that~\eqref{eq:fu1} is now replaced by
\begin{equation}\label{eq:fusing}
\|(f(u)-f(w))(s,\cdot)\|_{L^r}\leq C\,(1+s)^{-p\gamma_{pr}}\,\|u-w\|_{X(T)}\big( \|u\|_{X(T)}^{p-1}+\|w\|_{X(T)}^{p-1}\big),
\end{equation}
in particular we avoided the singularity at~$s=0$ in the singular problem. We shall distinguish estimates at short time and at a long time.

At short time~$t\leq1$, for~$q\in[p_\crit,\infty)$, we rely on the $L^q-L^q$ estimate~\eqref{eq:rqdecaygood} to get
\[ \begin{split}
\|(Fu-Fw)(t,\cdot)\|_{L^q}
    &  \leq C\,t^{(1-\mu)_+}\,\int_{t_0}^t s^{\min\{1,\mu\}}\,\| f(u(s,\cdot))-f(w(s,\cdot)) \|_{L^q}\,ds \\
    &  \leq C\,t^{(1-\mu)_+}\,\int_{t_0}^t s^{\min\{1,\mu\}}\,\,ds\,\|u-w\|_{X(T)}\big( \|u\|_{X(T)}^{p-1}+\|w\|_{X(T)}^{p-1}\big)\\
    &  \leq C_1\,\|u-w\|_{X(T)}\big( \|u\|_{X(T)}^{p-1}+\|w\|_{X(T)}^{p-1}\big)\,.
\end{split} \]
At long time~$t\geq1$, we proceed as in the proof of Theorem~\ref{thm:maingen}. In turn, this means that we shall prove the convergence of
\[ \int_{t_0}^\infty s^{\min\{1,\mu\}}\,(1+s)^{-p\gamma_p}\,\,ds, \]
and similarly for the other integrals. The condition for the convergence as~$s\to\infty$ is the same considered in the proof of Theorem~\ref{thm:maingen}. The convergence of the integral as~$s\to0$, or the fact that the estimate for the integral may be taken uniformly with respect to~$t_0$, is easily derived thanks to the fact that we fixed~$\alpha=0$, since the most singular powers of~$s$ which appears is~$s^{\frac{\mu-1}2}$, when~$\mu\in(0,1)$.

This concludes the proof of Theorem~\ref{thm:main1}.
\end{proof}
We stress that to extend the proof of Theorem~\ref{thm:main1} to the case of a nonlinearity of type~$t^{-\alpha}f(u)$, with~$\alpha>0$, it would be necessary to impose $\min\{1,\mu\}-\alpha>-1$ to get the convergence of the integral
\[ \int_0^1 s^{\min\{1,\mu\}-\alpha}\,\,ds, \]
and similar conditions for the other integrals appearing in the proof of Theorem~\ref{thm:maingen}.

\subsection{A result for~\eqref{eq:CPalpha} in space dimension~$n=2$}\label{sec:weakn2}

We may extend Theorem~\ref{thm:maingen} to space dimension~$n=2$ when~$\alpha\in(0,2)$ (Theorem~\ref{thm:main2} covers the case~$\alpha=0$), and this can be of interest in view of the results in Section~\ref{sec:Tricomi}. For the sake of brevity, we only consider weak solutions in~$L^\infty([t_0,\infty),L^p)$, with~$p\in(p_\crit,2)$ (the restriction~$p<2$ is made to use~$L^1-L^p$ estimates in Corollary~\ref{cor:main}). As in Remark~\ref{rem:findp}, we may compute the value of~$p_\crit$ in~\eqref{eq:pcritalpha} if~$n=2$.
\begin{remark}\label{rem:findp2}
Let us determine the value of~$p_\crit$ in~\eqref{eq:pcritalpha} if~$n=2$, according to the value of~$\mu>-1$. We stress that
\[ 1 + \frac{2-\alpha}{1+\min\{1,\mu\}}>p_\Str(2+\mu,\alpha) \iff 3-\max\{\mu,2-\mu\}<\frac2{p_\crit}\,. \]
It holds~$p_\crit=p_\Str(2+\mu,\alpha)$ if, and only if, $\alpha-1\leq\mu\leq\bar\mu$, where
\begin{equation}\label{eq:barmualpha2}
\bar \mu = 2 - \frac{\alpha}{4-\alpha}.
\end{equation}
It holds~$p_\crit=2-\alpha/2$ if, and only if, $\mu\geq\bar\mu$, and it holds~$p_\crit=1+(2-\alpha)/(1+\mu)$ if, and only if, $-1<\mu\leq\alpha-1$.
\end{remark}
\begin{proposition}\label{prop:gen2}
Fix~$n=2$, and either $\alpha\in(0,1]$ and $\mu>2(1-\alpha)$ or~$\alpha\in(1,2)$ and~$\mu>1-\alpha$. Let~$p\in(p_\crit,2)$, where~$p_\crit$ is as in~\eqref{eq:pcritalpha}. Then there exists~$\eps>0$ such that for any initial data
\begin{equation}\label{eq:dataL1}
u_1\in L^1, \quad \text{with}\ \|u_1\|_{L^1}\leq\eps,
\end{equation}
there exists a unique $u\in L^\infty([t_0,\infty),L^p)$, global-in-time weak solution to~\eqref{eq:CPalpha}. Moreover, for any~$\delta>0$, we have the decay estimate
\begin{equation}\label{eq:decayweak}
\|u(t,\cdot)\|_{L^p} \leq C\,\|u_1\|_{L^1}\times \begin{cases}
t^{(1-\mu)_+-2\left(1-\frac1q\right)} & \text{if~$3-2/q<\max\{\mu,2-\mu\}$,} \\
t^{\delta+\frac1q-\frac{\mu+1}2} & \text{if~$3-2/q\geq \max\{\mu,2-\mu\}$,}
\end{cases}
\end{equation}
where~$C=C(t_0)>0$.
\end{proposition}
\begin{proof}
We follow the proof of Theorem~\ref{thm:maingen}, but now, for a given~$T>t_0$, $X(T)$ is the subspace of~$L^\infty([t_0,\infty),L^p)$ for which
\[ \|v\|_{X(T)} = \begin{cases}
\sup_{t\in[t_0,T]} t^{2\left(1-\frac1p\right)-(1-\mu)_+} \|v(t,\cdot)\|_{L^p} & \text{if~$3-\max\{\mu,2-\mu\}< 2/p$} \\
\sup_{t\in[t_0,T]} t^{\frac{\mu+1}2-\frac1q-\delta} \|v(t,\cdot)\|_{L^p} & \text{if~$3-\max\{\mu,2-\mu\}\geq 2/p$,}
\end{cases} \]
for a sufficiently small~$\delta>0$ which we will fix later. We prove that there exists a constant~$C=C(t_0)$, independent of~$T$, such that the solution to the linear problem~\eqref{eq:CPlinear} with~$s=t_0$ and~$v_1=u_1$ verifies the estimate
\begin{equation}
\label{eq:dataX1}
\|v\|_{X(T)} \leq C\,\|u_1\|_{L^1}\,,
\end{equation}
and~\eqref{eq:contraX} holds. By standard contraction arguments (as in the proof of Theorem~\ref{thm:maingen}), properties \eqref{eq:dataX1} and~\eqref{eq:contraX} imply that there exists~$\eps>0$ such that if~$u_1$ verifies~\eqref{eq:dataL1}, then there is a unique global-in-time solution to~\eqref{eq:CPalpha}, verifying
\[ \|u\|_{X(T)} \leq C\,\|u_1\|_{L^1}\,, \]
for any~$T>t_0$, with~$C=C(t_0)$, independent of~$T$. It is clear that the solution to the linear problem~\eqref{eq:CPlinear} verifies~\eqref{eq:dataX1}, applying Corollary~\ref{cor:main} with~$r=1$ (here the assumption~$p<2$ comes into play). Indeed, if~$d(1,p)<\max\{\mu,2-\mu\}/2$, that is, $3-\max\{\mu,2-\mu\}< 2/p$, by~\eqref{eq:rqdecaygood}, with~$r=1$ and~$q=p$, we obtain
\begin{equation}\label{eq:1qdecaygoodn2}
\|v(t,\cdot)\|_{L^p} \leq C\,s^{\min\{1,\mu\}}\,t^{(1-\mu)_+-2\left(1-\frac1p\right)}\, \|v_1\|_{L^1} \,,
\end{equation}
for some~$C>0$, independent of~$s,t$. If~$d(1,p)\geq\max\{\mu,2-\mu\}/2$, that is, $3-\max\{\mu,2-\mu\}\geq2/p$, by~\eqref{eq:rqdecaybad}, with~$r=1$ and~$q=p$, we obtain
\begin{equation}\label{eq:1qdecaybadn2}
\|v(t,\cdot)\|_{L^q} \leq C\,s^{\frac1p+\frac{\mu-1}2-\delta}\,t^{\delta-\frac{\mu+1}2+\frac1q}\, \|v_1\|_{L^1}\,,
\end{equation}
taking~$\eps\leq\delta$. Setting~$s=t_0$, thanks to \eqref{eq:1qdecaygood}-\eqref{eq:1qdecaybad}, we get~\eqref{eq:dataX1}. Now let~$u,w\in X(T)$. We want to prove~\eqref{eq:contraX}.

If~$3-\max\{\mu,2-\mu\}< 2/p$, using~\eqref{eq:1qdecaygoodn2}, we obtain
\begin{align*}
& t^{2\left(1-\frac1p\right)-(1-\mu)_+} \, \|(Fu-Fw)(t,\cdot)\|_{L^p}\\
    & \qquad \leq C\, \int_{t_0}^t s^{\min\{1,\mu\}-\alpha}\,\| (f(u)-f(w))(s,\cdot) \|_{L^1}\,ds \\
    & \qquad \leq C\,\int_{t_0}^t s^{\min\{1,\mu\}-\alpha-p\left(2\left(1-\frac1p\right)-(1-\mu)_+\right)}\,\,ds\,\|u-w\|_{X(T)}\big( \|u\|_{X(T)}^{p-1}+\|w\|_{X(T)}^{p-1}\big)\\
    & \qquad \leq C_1\,\|u-w\|_{X(T)}\big( \|u\|_{X(T)}^{p-1}+\|w\|_{X(T)}^{p-1}\big),
\end{align*}
where the fact that the integral
\[ \int_{t_0}^\infty s^{\min\{1,\mu\}-\alpha-p\left(2\left(1-\frac1p\right)-(1-\mu)_+\right)}\,\,ds \]
is convergent is a consequence of~$p>1+(2-\alpha)/(1+\min\{1,\mu\})$. If~$3-\max\{\mu,2-\mu\}\geq 2/p$, using~\eqref{eq:1qdecaybadn2}, we obtain
\begin{align*}
& t^{\frac{\mu+1}2-\frac1p-\delta} \, \|(Fu-Fw)(t,\cdot)\|_{L^p}\\
    & \qquad \leq C\, \int_{t_0}^t s^{-\alpha+\frac1p+\frac{\mu-1}2-\delta}\,\| f(u(s,\cdot))-f(w(s,\cdot)) \|_{L^1}\,ds \\
    & \qquad \leq C\,\int_{t_0}^t s^{-\alpha+\frac1p+\frac{\mu-1}2-\delta-p\left(\frac{\mu+1}2-\frac1p-\delta\right)}\,\,ds\,\|u-w\|_{X(T)}\big( \|u\|_{X(T)}^{p-1}+\|w\|_{X(T)}^{p-1}\big)\\
    & \qquad C_1\,\|u-w\|_{X(T)}\big( \|u\|_{X(T)}^{p-1}+\|w\|_{X(T)}^{p-1}\big),
\end{align*}
where the fact that the integral
\[ \int_{t_0}^\infty s^{-\alpha+\frac1p+\frac{\mu-1}2-\delta-p\left(\frac{\mu+1}2-\frac1p-\delta\right)}\,\,ds\]
is convergent is a consequence of~$p>p_\Str(2+\mu,\alpha)$, for a sufficiently small~$\delta>0$. This concludes the proof of Proposition~\ref{prop:gen2}.
\end{proof}

%%%%%%%%%%%%%%%%%%%%%%%%%%%%%%%%%%%%%%%%%%%%%%%%%%%%%%%%%%%%%%%%%%%%%

\section{Energy solutions for semilinear Tricomi generalized equations}\label{sec:Tricomi}

The strictly hyperbolic semilinear Cauchy problem for the generalized Tricomi equation
\begin{equation}\label{eq:CPwave}
\begin{cases}
w_{tt}- t^{2\ell}\,\triangle w = f(w), & t\geq t_1>0, \ x\in\R\,,\\
w(t_1,x)=0\,, \quad w_t(t_1,x)=w_1(x)\,,
\end{cases}
\end{equation}
where~$\ell>0$, is equivalent to the regular Cauchy problem~\eqref{eq:CPalpha} for the Euler-Poisson-Darboux equation with parameter~$\mu=\ell/(\ell+1)$, and power nonlinearity~$t^{-2\mu}f(u)$. Indeed, by the change of variable
\begin{equation}\label{eq:changewave}
w(t,x)=u(\Lambda(t),x),\quad \text{where}\ \Lambda(t)=\frac{t^{\ell+1}}{\ell+1}\,,
\end{equation}
noticing that~$w_t = t^\ell\,u_t(\Lambda(t),x)$, then problem~\eqref{eq:CPalpha} is equivalent to~\eqref{eq:CPwave}, with~$w_1=t_1^\ell\, u_1$, provided that
\begin{equation}\label{eq:changealpha}
\mu=\frac\ell{\ell+1},\quad \alpha=2\mu,
\end{equation}
and where
\[ %c_\ell = (\ell+1)^{-\alpha}, \quad
t_1=\Lambda^{-1}(t_0)=((\ell+1)t_0)^{\frac1{\ell+1}}. \]
Hence, in space dimension~$n=1$, we have the following consequence of Theorem~\ref{thm:maingen}.
\begin{corollary}\label{cor:waves}
Let~$n=1$ and~$\ell>0$, and assume that~$p>p_\crit=1+2/\ell$. Then there exists~$\eps>0$ such that for any
\begin{equation}\label{eq:dataell}
w_1\in L^1\cap L^2, \quad \text{with}\ \|w_1\|_{L^1}+\|w_1\|_{L^2}\leq\eps,
\end{equation}
there exists a unique $w\in\mathcal C([t_1,\infty),H^1)\cap\mathcal C^1([t_1,\infty),L^2)$, global-in-time energy solution to~\eqref{eq:CPwave}. Moreover, we have the energy estimate
\begin{equation}\label{eq:energyestell}
E(t) = \frac12\,\|w_t(t,\cdot)\|_{L^2}^2 + \frac12\,t^{2\ell}\,\|w_x(t,\cdot)\|_{L^2}^2 \leq C\,t^{\ell}\big(\|w_1\|_{L^1}^2+\|w_1\|_{L^2}^2\big),
\end{equation}
it holds~$w\in L^\infty([t_1,\infty),L^{p_\crit})$ if~$\ell>2$, and for any~$\delta>0$, the following decay estimate holds:
\begin{equation}\label{eq:decayell}
\|w(t,\cdot)\|_{L^q} \leq C\,\big(\|w_1\|_{L^1}+\|w_1\|_{L^2}\big)\times \begin{cases}
t^{-\ell+\frac{\ell+1}q} & \text{if~$q\in[1+2/\ell,2+2/\ell)$,}\\
t^{\delta-\frac\ell{2}} & \text{if~$2+2/\ell\leq q<\infty$,}
\end{cases}
\end{equation}
where~$C=C(t_1)>0$.
\end{corollary}
Similarly, Proposition~\ref{prop:gen2} gives the existence of global-in-time weak solutions to~\eqref{eq:CPwave} in space dimension~$n=2$, for~$\ell>2/3$. In this case~$p_\crit=p_\Str(1+\mu,\alpha)=p_\Str((2\ell+1)/(\ell+1),2\ell/(\ell+1))$ (see Remark~\ref{rem:findp2}).
\begin{remark}\label{rem:Tricomi}
The singular Cauchy problem for EPD equation with parameter~$\mu=\ell/(\ell+1)$ and nonlinearity~$t^{-2\mu}f(u)$ is equivalent to the semilinear weakly hyperbolic problem
\begin{equation}\label{eq:CPwavewhyp}
\begin{cases}
w_{tt}- t^{2\ell}\,\triangle w = f(w), & t>0, \ x\in\R\,,\\
w(0,x)=w_0(x)\,, \quad w_t(0,x)=0\,.
\end{cases}
\end{equation}
In space dimension~$n=1$, the global-in-time existence for~$p>1+2/\ell$ has been proved for small data weak solutions to~\eqref{eq:CPwavewhyp}, in the recent paper~\cite{HWY20}, see also~\cite{Ga}. The nonexistence of global-in-time weak solutions to~\eqref{eq:CPwave} or~\eqref{eq:CPwavewhyp} for~$p\in(1,1+2/\ell]$ in space dimension~$n=1$ %(and, more in general, for~$p\in(1,1+2/(n(\ell+1)-1)]$ in space dimension~$n\geq1$),
under suitable sign condition on~$w_1$ or~$w_0$, is proved in Theorem 3.1 in~\cite{DAL03}. In space dimension~$n\geq2$, in~\cite{HWY17a,HWY17} it is proved that the critical exponent for global small data weak solutions to~\eqref{eq:CPwavewhyp}, is
\[ \frac{(n-1)(\ell+1)-\ell}2\,(p_\crit-1)-(1-\ell)-\frac{\ell+1}{p_\crit}=0. \]
We stress that~$p_\crit=p_\Str(n+\mu-1,\alpha)$, where~$p_\Str$ is as in~\eqref{eq:Straussmod}, for~$\mu=\ell/(\ell+1)$ and~$\alpha=2\mu$.
\end{remark}
We may directly prove Corollary~\ref{cor:waves} as a consequence of Theorem~\ref{thm:maingen}, with~$\mu\in(0,1)$ and~$\alpha=2\mu$.
\begin{proof}[Proof of Corollary~\ref{cor:waves}]
We fix~$\mu$ and~$\alpha$ as in~\eqref{eq:changealpha} and we apply Theorem~\ref{thm:maingen}. Due to~$\mu\in(0,1)$, according to Remark~\ref{rem:findp}, we get that
\[ p_\crit=1+\frac{2-\alpha}\mu=1+\frac2\ell.\]
%
%if~$n=1$, and
%%
%\[ p_\crit=p_\Str(2+\mu,\alpha)=p_\Str(2+\ell/(\ell+1),2\ell/(\ell+1)),\]
%
%if~$n=2$. We also notice that the condition~$\mu\geq2(1-\alpha)$ if~$n=2$, gives us~$\mu\geq2/5$, that is, $\ell\geq2/3$.
Recalling that~$w_t = t^\ell\,u_t(\Lambda(t),x)$, by the energy estimate~\eqref{eq:energyest} we deduce
\begin{align*}
\|w_t(t,\cdot)\|_{L^2}^2 + t^{2\ell}\,\|w_x(t,\cdot)\|_{L^2}^2
    & = t^{2\ell} \big( \|u_t(\Lambda(t),\cdot)\|_{L^2}^2 + \|u_x(\Lambda(t),\cdot)\|_{L^2}^2 \big) \\
    & \leq C\,t^{2\ell} \,\Lambda(t)^{-\mu} = C_1\,t^{\ell}\,,
\end{align*}
so that we derive~\eqref{eq:energyestell}. Similarly, by~\eqref{eq:decay} we obtain
\begin{align*}
\|w(t,\cdot)\|_{L^q}
    & = \|u(\Lambda(t),\cdot)\|_{L^q} \\
    & \leq \begin{cases}
C\,\Lambda(t)^{-\mu+\frac1q}\,\big(\|v_1\|_{L^1}+\|v_1\|_{L^2}\big) & \text{if~$q\in[1+2/\ell,2+2/\ell)$,}\\
C\,\Lambda(t)^{\delta-\frac\mu2}\,\big(\|v_1\|_{L^1}+\|v_1\|_{L^2}\big) & \text{if~$2+2/\ell\leq q<\infty$.}
%\Lambda(t)^{\delta+\frac1q-\frac{\mu+1}2}\,\big(\|v_1\|_{L^1}+\|v_1\|_{L^2}\big) & \text{if~$n=2$ and~$q\in[p_\crit,6]$.}
\end{cases}
\end{align*}
Replacing
\[ \Lambda(t)^{-\frac\mu2}= c_1\,t^{-\frac\ell{2}}, \quad \Lambda(t)^{-\mu+\frac1q}= c_2\,t^{-\ell + \frac{\ell+1}{q}},\quad \Lambda(t)^{\frac1q-\frac12-\frac{\mu}2}=c_3\,t^{-\frac\ell2+(\ell+1)\left(\frac1q-\frac12\right)}, \]
we get~\eqref{eq:decayell}. This concludes the proof.
\end{proof}
More in general, Theorem~\ref{thm:maingen} may be applied to study semilinear waves with increasing polynomial speed of propagation and critical dissipation. Indeed, problem
\begin{equation}\label{eq:CPwavegen}
\begin{cases}
w_{tt}- t^{2\ell}\,\triangle w + \dfrac\nu{t}\, w_t = f(w), & t\geq t_1, \ x\in\R^n\,,\\
w(t_1,x)=0\,, \quad w_t(t_1,x)=w_1(x)\,,
\end{cases}
\end{equation}
with~$\ell>0$ and~$\nu>-\ell$, is equivalent to problem~\eqref{eq:CPalpha} with %%% SE NON FAI ELL>0 NON VIENE ALPHA POSITIVO
\begin{equation}\label{eq:changealphagen}
\mu=\frac{\nu+\ell}{\ell+1},\quad \alpha=\frac{2\ell}{\ell+1}.
\end{equation}
%
%We mention that condition~$\nu>-\ell$ corresponds to the dissipative condition obtained by taking into account of variable speed of propagation~\cite{DAE13}.
%
\begin{proposition}\label{prop:wavesdiss}
Let~$n=1$ and~$\ell>0$. %Assume that~$n=1$ and~$\nu>-\ell$, or~$n=2$ and~$\nu\geq 2-3\ell$ if~$\ell\in(-1,1]$ or~$\nu\geq1-2\ell$ if~$\ell>1$.
Applying Theorem~\ref{thm:maingen} to~\eqref{eq:CPwavegen}, we find that global-in-time small data energy solutions exist for~$p>p_\crit$, where we distinguish three cases:
\begin{itemize}
\item if~$\mu\geq\max\{1,\bar\mu\}$, that is, $\nu\geq\max\{1,\bar\mu\}$, where
\[ \bar\nu = -\ell + \frac{4(\ell+1)}{3+\ell}, \]
%\[ \nu \geq -\ell + (\ell+1)\Big( n-1 + \frac4{(n+2)(1+\ell)-2\ell}\Big), \]
%
then
\[ p_\crit = 1+ \frac2{1+\ell}; \]
\item if~$0<\mu\leq\min\{1,\alpha\}$, that is, $-\ell<\nu\leq\min\{1,\ell\}$, or~$\alpha\in[1,2)$ and~$\mu\leq1$, that is, $-\ell<\nu\leq1\leq\ell$ then
\[ p_\crit = 1 + \frac2\ell; \]
%
%\item if~$n=2$, $\mu\leq\alpha-1$, that is, $\nu\leq-1$, then
%%
%\[ p_\crit = 1 + \frac2{1+2\ell}; \]
%
\item if~$\alpha\in[0,1)$ and~$\alpha<\mu<\bar\mu$, that is, $\ell<1$ and $\ell<\nu<\bar\nu$, then $p_\crit$ is the solution to~\eqref{eq:Straussmod}, i.e.
\[ \frac{\nu+\ell}2(p-1)+\ell-1-\frac{\ell+1}{p}=0.\]
\end{itemize}
Energy estimates and estimates for~$\|w(t,\cdot)\|_{L^q}$, with~$q\in[p_\crit,\infty)$, are derived accordingly by~\eqref{eq:energyest} and~\eqref{eq:decay}, as we did in the proof of Corollary~\ref{cor:waves}. We omit the details for brevity.
\end{proposition}
Some results about semilinear waves with time-dependent speed of propagation and effective dissipation are collected in~\cite{BR13+}; roughly speaking, they shall correspond to take~$\mu=\infty$.

\section{Proof of Theorem~\ref{thm:main2}}\label{sec:proof2}

In space dimension~$n\geq3$, it is not possible to apply Corollary~\ref{cor:main} with~$r=1$, due to~$d(1,q)>1$ for any~$q>1$. The same is true in space dimension~$n=2$, for any~$q>2$. For this reason, we use Corollary~\ref{cor:main2} to prove Theorem~\ref{thm:main2}.

%As a consequence we cannot recover a range of critical exponents for all possible values of~$\mu$. Still, we may prove global-in-time existence of small data energy solutions to~\eqref{eq:CP} for any~$p>1+2/n$ (and~$p\leq 1+2/(n-2)$, due to the Sobolev embedding~$H^1(\R^n)\hookrightarrow L^{2p}(\R^n)$) for sufficiently large~$\mu$, namely, for~$\mu\geq n$, in space dimension~$n=3,4,5$. This result improves the threshold~$n+2$, previously known for the existence result~\cite{DA15}, arriving much closer to the theoretical threshold~$n-1+4/(n+2)$. The exponent is still critical, but it is not clear if the threshold may be lowered or not, possibly considering solutions in different functional spaces.
%

%

%
%We now want to compute
%%
%\[ \frac12-\frac1q-p\gamma_{r_2(q)p}. \]
%%
%We notice that, due to
%%
%\[ \frac2{r_2(q)} = \frac32+\frac1q, \]
%%
%and~$q\geq2$, $p_\crit\leq2$, it follows that
%%
%\[ \frac1{p_\crit r_2(q)} - \frac1q \geq \frac1{2r_2(q)}-\frac1q = \frac38-\frac3{4q}\geq0, \]
%%
%that is, $r_2(q)p_\crit\leq q$. As a consequence, $3-\mu<2/q$ implies $3-\mu<2/(r_2(q)p_\crit)$, so that
%%
%\begin{equation}\label{eq:r2q}
%\frac12-\frac1q-p\gamma_{r_2(q)p} < \frac12-\frac1q-p_\crit\gamma_{r_2(q)p_\crit} = \frac12-\frac1q-2\left(p_\crit-\frac1{r_2(q)}\right) = 2(1-p_\crit).
%\end{equation}
%%
%Therefore, the integral
%%
%\[ \int_{t_0}^\infty s^{1-\alpha}\,\big(s^{-p\gamma_p}+s^{\frac12-\frac1q-p\gamma_{r_2(q)p}}\big)\,\,ds, \]
%%
%is convergent.

\begin{proof}[Proof of Theorem~\ref{thm:main2}]
As in the proof of Theorem~\ref{thm:maingen}, for a general~$T>t_0$, we fix~$X(T)$ as a subspace of the energy space~$\mathcal C([t_0,T],H^1)\cap\mathcal C^1([t_0,T],L^2)$ if~$n=2$, and of~$\mathcal C([t_0,T],H^1)\cap\mathcal C^1([t_0,T],L^2)\cap L^\infty([t_0,T],L^{p_\crit})$ if~$n\geq3$, and we prove that there exists a constant~$C=C(t_0)$, independent of~$T$, such that~\eqref{eq:dataX} and~\eqref{eq:contraX} hold. Properties \eqref{eq:dataX}-\eqref{eq:contraX} imply that there exists~$\eps>0$ such that if~$v_1$ verifies~\eqref{eq:data}, then there is a unique global-in-time solution to~\eqref{eq:CP}, verifying
\[ \|u\|_{X(T)} \leq C\,\big(\|v_1\|_{L^1}+\|v_1\|_{L^2}\big)\,, \]
for any~$T>t_0$, with~$C=C(t_0)$, independent of~$T$. We recall that~$p_\crit=1+2/n$.

We fix~$g(t)$ as in~\eqref{eq:eng}, and we set
\begin{equation}\label{eq:gammaq1n}
\forall q\in[p_\crit,2+4/(n-1)]: \ \gamma_q=\begin{cases}
n\left(1-1/q\right) & \text{if~$\mu>n+1-2/q$,}\\
(\mu+n-1)/2-(n-1)/q-\delta & \text{if~$\mu\leq n+1-2/q$,}
\end{cases}
\end{equation}
for a sufficiently small~$\delta>0$, which we will fix later. For the ease of notation, we also define
\[ \gamma_q = \gamma_{2+4/(n-1)}, \qquad \forall q\in \Big(2+\frac4{n-1}, 2+\frac4{n-2}\Big]. \]
We fix
\begin{equation}\label{eq:normaXp}
\|u\|_{X(T)} = \sup_{t\in[t_0,T]} \Bigl( (g(t))^{-1}\, \|(u_t,\nabla u)(t,\cdot)\|_{L^2} + \sup \big\{ t^{\gamma_q}\,\|u(t,\cdot)\|_{L^q}: \ q\in[p_\crit,2+4/(n-2)] \big\} \Bigr).
\end{equation}
With this choice of norm on~$X(T)$, the solution to the linear problem~\eqref{eq:CPlinear} verifies~\eqref{eq:dataX}. Indeed, due to~$q\leq 2(n+1)/(n-1)$, we may apply Corollary~\ref{cor:main2} with~$r_2=r_2(q)$, verifying~$d(r_2(q),q)=1$.

If~$p_\crit\leq q<2/(n+1-\mu)$, we obtain
\begin{equation}\label{eq:1qdecaygoodn}
\|v(t,\cdot)\|_{L^q} \leq C\,s\,t^{-n\left(1-\frac1q\right)+\delta}\, \big(\|v_1\|_{L^1}+s^{\frac{n-1}2-\frac1q}\,\|v_1\|_{L^{r_2}}\big) \,,
\end{equation}
for some~$C>0$, independent of~$s,t$. If~$q\geq2/(n+1-\mu)$, taking~$\eps\leq\delta$, we obtain:
\begin{equation}\label{eq:1qdecaybadn}
\|v(t,\cdot)\|_{L^q} \leq C\,s^{\frac\mu2-\delta}\,t^{\delta-(n-1)\left(\frac12-\frac1q\right)-\frac\mu2}\, \big(s^{-\frac{n-1}2+\frac1q}\|v_1\|_{L^1}+\| v_1\|_{L^{r_2}}\big)\,.
\end{equation}
Setting~$s=t_0$, thanks to \eqref{eq:1qdecaygoodn}, \eqref{eq:1qdecaybadn}, and Proposition~\ref{prop:energy}, we get~\eqref{eq:dataX}.

Now let~$u,w\in X(T)$ and~$q\in[p_\crit,2+4/(n-1)]$.

If~$p_\crit\leq q<2/(n+1-\mu)$, we obtain
\begin{equation}\label{eq:1qdecayFgoodn}
\begin{split}
& t^{n\left(1-\frac1q\right)-\delta}\,\|(Fu-Fw)(t,\cdot)\|_{L^q} \\
    & \quad \leq C\,\int_{t_0}^t s\,\big(\| (f(u)-f(w))(s,\cdot) \|_{L^1}+ s^{\frac{n-1}2-\frac1q}\,\| (f(u)-f(w))(s,\cdot) \|_{L^{r_2(q)}}\big)\,ds,
\end{split}\end{equation}
for some~$C>0$, independent of~$t_0,t$. If~$q\geq2/(n+1-\mu)$, taking~$\eps\leq\delta$, we obtain:
\begin{equation}\label{eq:1qdecayFbadn}
\begin{split}
& t^{-\delta+(n-1)\left(\frac12-\frac1q\right)-\frac\mu2}\,\|(Fu-Fw)(t,\cdot)\|_{L^q}\\
    &\quad \leq C\,\int_{t_0}^t s^{\frac\mu2-\delta}\, \big(s^{-\frac{n-1}2+\frac1q}\|(f(u)-f(w))(s,\cdot)\|_{L^1}+\|(f(u)-f(w))(s,\cdot)\|_{L^{r_2(q)}}\big)\,.
\end{split}\end{equation}
Using~\eqref{eq:fu} with H\"older inequality, and the fact that~$u,w\in X(T)$, we may estimate~$\|(f(u)-f(w))(s,\cdot)\|_{L^1}$ as in~\eqref{eq:fu1}, and we may estimate
\begin{equation}\label{eq:fur}
\begin{split}
\|(f(u)-f(w))(s,\cdot)\|_{L^{r_2(q)}}
    & \leq C\, \|(u-w)(s,\cdot)\|_{L^{r_2(q)p}}\,\big( \|u(s,\cdot)\|_{L^{r_2(q)p}}^{p-1}+\|w(s,\cdot)\|_{L^{r_2(q)p}}^{p-1}\big) \\
    & \leq C\,s^{-p\gamma_{r_2(q)p}}\,\|u-w\|_{X(T)}\big( \|u\|_{X(T)}^{p-1}+\|w\|_{X(T)}^{p-1}\big).
\end{split}\end{equation}
If~$p_\crit\leq q<2/(n+1-\mu)$, we want to prove that the integral
\begin{equation}\label{eq:intgood}
\int_{t_0}^\infty s\,\big( s^{-p\gamma_p} + s^{\frac{n-1}2-\frac1q-p\gamma_{r_2(q)p}} \big) \,ds
\end{equation}
is convergent. If~$q\geq2/(n+1-\mu)$, we want to prove that the integral
\begin{equation}\label{eq:intbad}
\int_{t_0}^\infty s^{\frac\mu2-\delta}\,\big( s^{-\frac{n-1}2+\frac1q-p\gamma_p} + s^{-p\gamma_{r_2(q)p}} \big) \,ds
\end{equation}
is convergent.

It is now crucial to remark that the property that~$r_2(q)p_\crit\leq 2+4/(n-1)$ for any~$q\leq 2+4/(n-1)$ is true, since this property guarantees that~$\gamma_{r_2(q)p_\crit}$ is described by~\eqref{eq:gammaq1n}. The property is true, since:
\[ q\leq \frac{2(n+1)}{n-1}\Rightarrow r_2(q)\leq \frac{2(n+1)}{n+3} \Rightarrow r_2(q)\frac{n+2}n < \frac{2(n+1)}{n-1}\,. \]
First we consider~$q<2/(n+1-\mu)$ and we prove that~\eqref{eq:intgood} is convergent. It is clear that
\[ p\gamma_p>p_\crit\,\gamma_{p_\crit} = n(p_\crit-1) = 2.\]
On the other hand, due to~$d(r_2,q)=1$, that is,
\[ \frac1{r_2}=\frac1n\left(\frac{n+1}2+\frac1q\right), \]
we find
\[ \frac1{p_\crit r_2}-\frac1q=\frac1{n+2}\left(\frac{n+1}2+\frac1q\right)-\frac1q=\frac{n+1}{n+2}\left(\frac12-\frac1q\right), \]
that is, $p_\crit r_2(q)\leq q$ if, and only if, $q\geq2$. In particular, $p_\crit r_2(q)\leq \max\{2,q\}$. As a consequence, since we assumed~$\mu\geq n$, and~$q<2/(n+1-\mu)$, we also get
\[ \gamma_{r_2(q)p_\crit}=n\left(1-\frac1{r_2(q)p_\crit}\right).\]
Using~$d(r_2(q),q)=1$, we obtain
\begin{align*}
\frac{n-1}2-\frac1q-p\gamma_{r_2(q)p}
    & < \frac{n-1}2-\frac1q-p_\crit\gamma_{r_2(q)p_\crit} \\
    & =\frac{n-1}2-\frac1q - n(p_\crit-1)-n\left(1-\frac1{r_2(q)}\right) \\
    & = - n(p_\crit-1) = -2.
\end{align*}
This proves that~\eqref{eq:intgood} is convergent.

Now we consider~$q\geq 2/(n+1-\mu)$ and we prove that~\eqref{eq:intbad} is convergent. It is clear that
\[ \frac\mu2-\delta-\frac{n-1}2+\frac1q-p\gamma_p < \frac\mu2-\frac{n-1}2+\frac1q-2 \leq -1, \]
for a sufficiently small~$\delta$, due to~$p\gamma_p>2$. We distinguish two subcases. If
\[ n+1-\frac2{r_2(q)p_\crit}< \mu \leq n+1-\frac2q, \]
then
\begin{align*}
\frac\mu2-\delta-p\gamma_{r_2(q)p}
    &  < \frac\mu2 -p_\crit \gamma_{r_2(q)p_\crit} = \frac\mu2 -2-n\left(1-\frac1{r_2(q)}\right) \\
    & \leq -2 -\frac{n-1}2-\frac1q+\frac{n}{r_2(q)}=-1.
\end{align*}
On the other hand, if
\[ \mu\leq n+1-\frac2{r_2(q)p_\crit}, \]
then
\[ \frac\mu2-\delta-p\gamma_{r_2(q)p} < \frac\mu2 -p_\crit \left(\frac{\mu+n-1}2-\frac{n-1}{r_2(q)p_\crit}\right), \]
for a sufficiently small~$\delta$. We may estimate
\begin{align*}
\frac\mu2 -p_\crit \left(\frac{\mu+n-1}2-\frac{n-1}{r_2(q)p_\crit}\right)
    & = \frac\mu2 -\frac{n+2}n \, \frac{\mu+n-1}2 + \frac{n-1}{r_2(q)}\\
    & = -\frac\mu{n}-\frac{n-1}n\left(\frac12-\frac1q\right)\leq-1,
\end{align*}
where in the last inequality we used~$\mu\geq n$ and~$q\geq2$.

This proves that~\eqref{eq:intbad} is convergent.

Summarizing, so far we proved that
\[ \|(Fu-Fw)(t,\cdot)\|_{L^q} \leq C(t_0)\,t^{-\gamma_q}\,\|u-w\|_{X(T)}\big( \|u\|_{X(T)}^{p-1}+\|w\|_{X(T)}^{p-1}\big), \]
for any~$q\in[p_\crit,2(n+1)/(n-1)]$.

We now consider the energy estimates. Exception given for the case~$\mu=n+2$, for which an additional logarithmic power of~$s$ appears in the integral, by Proposition~\ref{prop:energy}, we obtain:
\begin{align*}
& (g(t))^{-1}\, \|(\nabla,\partial_t)(Fu-Fw)(t,\cdot)\|_{L^2}\\
    & \qquad \leq C\,\int_{t_0}^t s^{\min\{1,\frac{\mu-n}2\}}\,\big(\| (f(u)-f(w))(s,\cdot) \|_{L^1}+s^{\frac{n}2}\,\| (f(u)-f(w))(s,\cdot) \|_{L^2}\big)\,ds \\
    & \qquad \leq C\,\int_{t_0}^t \big( s^{\frac{\min\{2,\mu-n\}}2-p\gamma_p}+s^{\frac{\min\{n+2,\mu\}}2-p\gamma_{2p}}\big)\,\,ds\,\|u-w\|_{X(T)}\big( \|u\|_{X(T)}^{p-1}+\|w\|_{X(T)}^{p-1}\big).%\\
%    & \leq C(t_0)\,g(t)\,\|u-w\|_{X(T)}\big( \|u\|_{X(T)}^{p-1}+\|w\|_{X(T)}^{p-1}\big),
\end{align*}
It is clear that
\[ \int_{t_0}^\infty s^{\frac{\min\{2,\mu-n\}}2-p\gamma_p}\,ds, \]
is convergent, due to~$p\gamma_p>2$. On the other hand, if~$n+2/(n+2)<\mu$, then we immediately get
\begin{align*}
\frac{\min\{n+2,\mu\}}2-p\gamma_{2p}
    & < \frac{\min\{n+2,\mu\}}2-p_\crit\gamma_{2p_\crit} \\
    & = \begin{cases}
-1 & \text{if~$\mu\geq n+2$,}\\
\frac\mu2 - \frac{n+4}2 <-1 & \text{if~$n+2/(n+2)<\mu<n+2$,}
\end{cases}
\end{align*}
whereas, if~$\mu\leq n+2/(n+2)$, then, for a sufficiently small~$\delta$,
\begin{align*}
\frac{\min\{n+2,\mu\}}2-p\gamma_{2p}
    & < \frac\mu2 - p_\crit \left( \frac{\mu+n-1}2-\frac{n-1}{2p_\crit} \right)\\
    & = -\frac\mu{n} - \frac{n-1}n < -1.
\end{align*}
Therefore,
\[ \|(\nabla,\partial_t)(Fu-Fw)(t,\cdot)\|_{L^2}\leq C(t_0)\,g(t)\,\|u-w\|_{X(T)}\big( \|u\|_{X(T)}^{p-1}+\|w\|_{X(T)}^{p-1}\big).\]
Summarizing, we proved~\eqref{eq:contraX}, and this concludes the proof of Theorem~\ref{thm:main2}.
\end{proof}

%%%%%%%%%%%%%%%%%%%%%%%%%%%%%%%%%%%%%%%%%%%%%%%%%

\section{Proof of Theorem~\ref{thm:L2}}\label{sec:proofL2}

Here we prove Theorem~\ref{thm:L2}. The key to find global-in-time (weak) solutions in large space dimension is related to the use $L^2-L^q$ estimates for the solution to the linear problem~\eqref{eq:CPlinear}, which feels the restriction of initial datum to be ``only'' in~$L^2$, and more general $L^r-L^q$ estimates, where~$r=r(q)$, to estimate~$Fu$.
\begin{remark}\label{rem:weak}
Before proving Theorem~\ref{thm:L2}, we motivate the choice of the solution regularity $u\in L^\infty([t_0,\infty),L^{q_0}\cap L^{q_1})$, where
\[
q_0=2+\frac4{n+1}, \quad q_1=2+\frac4{n-1}\,,
\]
as in~\eqref{eq:q0q1}, and of the restriction~$p\leq q_1-1 = 1+4/(n-1)$. We define~$r=r(q)$ as the solution to~$d(r(q),q)=1$. Due to~$d(q_1',q_1)=1$, we find that
\begin{equation}\label{eq:rq}
\frac{n}{r(q)} = \frac{n+1}2+\frac1q,
\end{equation}
for any~$q\in [2, q_1]$. In this sense, the number $q_1$ is the largest exponent~$q$ such that there exists~$r\leq q'$ with~$d(r,q)\leq1$.

The restriction~$p\leq q_1-1$ is used to have~$r(q_1)p\leq r(q_1)(q_1-1)=q_1'(q_1-1)=q_1$, for any~$p\leq q_1-1$. As a consequence, it is clear that
\[ \forall p\leq q_1-1, \quad \forall q\in[2,q_1]: \quad r(q)p \leq q_1. \]
Similarly, the choice of~$q_0$ is motivated by~$p_\crit=1+4/n$. Indeed, $q_0$ is chosen so that~$r(q_0)p_\crit=q_0$, due to
\[ \frac1{r(q_0)p_\crit} = \frac1{n+4}\,\frac{n}{r(q_0)} = \frac1{n+4}\left( \frac{n+1}2+\frac1{q_0} \right) = \frac{n+1}{n+4} \left(\frac12+\frac1{2(n+3)}\right)=\frac1{q_0}\,. \]
Moreover, due to the fact that~$q/r(q)$ is an increasing function with respect to~$q$ (since~$(q/r(q))'=(n+1)/(2n)>0$), we also deduce that
\begin{equation}\label{eq:rqp}
r(q)p_\crit \leq q.
\end{equation}
In particular, we proved that for any~$p\in[1+4/n,1+4/(n-1)]$ and for any~$q\in[q_0,q_1]$, it holds~$r(q)p\in[q_0,q_1]$.
\end{remark}
\begin{proof}[Proof of Theorem~\ref{thm:L2}]
As in the proof of Theorem~\ref{thm:maingen}, for a general~$T>t_0$, we fix~$X(T)$ as a subspace of the space~$L^\infty([t_0,T],L^{q_0}\cap L^{q_1})$, and we prove that there exists a constant~$C=C(t_0)$, independent of~$T$, such that we have
\begin{equation}\label{eq:dataX2}
\|v\|_{X(T)} \leq C\,\|v_1\|_{L^2}\,,
\end{equation}
and~\eqref{eq:contraX} holds.

Properties \eqref{eq:dataX2} and~\eqref{eq:contraX} imply that there exists~$\eps>0$ such that if~$v_1$ verifies~\eqref{eq:data2}, then there is a unique global-in-time solution to~\eqref{eq:CP}, verifying
\[ \|u\|_{X(T)} \leq C\,\|v_1\|_{L^2}\,, \]
for any~$T>t_0$, with~$C=C(t_0)$, independent of~$T$.

We assume that~$\mu>1$ if~$n=3$, postponing this exceptional case later on. For any~$q\in[q_0,q_1]$, we set
\[ \gamma_q = \frac12\,\min\{n(1-2/q),\mu\} = \begin{cases}
n\left(\frac12-\frac1q\right) & \text{if~$\mu\geq2$ or~$n(1-2/q)\leq\mu$,}\\
\frac\mu2 & \text{if~$\mu\in(1,2)$ and~$n(1-2/q)\geq\mu$.}
\end{cases} \]
Let~$X(T)$ be the subspace of functions in~$L^\infty([t_0,T],L^{q_0}\cap L^{q_1})$, verifying
\begin{equation}\label{eq:normaX2}
\|u\|_{X(T)} = \sup \big\{ t^{\gamma_q}\,\|u(t,\cdot)\|_{L^q}: \ q\in[q_0,q_1], \quad t\in[t_0,T] \big\} \Bigr).
\end{equation}
With this choice of norm on~$X(T)$, the solution to the linear problem~\eqref{eq:CPlinear} verifies~\eqref{eq:dataX2}. Indeed, we may apply Corollary~\ref{cor:main} with~$r=2$, and we get
\begin{equation}\label{eq:2qdecay}
\|v(t,\cdot)\|_{L^q} \leq C\,\|v_1\|_{L^2} \,\times\begin{cases}
s\,t^{-n\left(\frac12-\frac1q\right)} & \text{if~$\mu\geq2$ or~$n(1-2/q)\leq\mu$,}\\
s^{1-n\left(\frac12-\frac1q\right)+\frac\mu2}\,t^{-\frac\mu2} & \text{if~$\mu\in(1,2)$ and~$n(1-2/q)\geq\mu$,}
\end{cases}\end{equation}
for some~$C>0$, independent of~$s,t$. Setting~$s=t_0$, we find~\eqref{eq:dataX2}.

Now let~$u,w\in X(T)$, $q\in[q_0,q_1]$, and~$r(q)$ as in~\eqref{eq:rq}. Due to the fact that~$r(q)\leq q_1'<2$ for any~$q\in[q_0,q_1]$, using~$s\leq t$ in~\eqref{eq:rqdecaygood}, we may obtain the same estimate as in~\eqref{eq:2qdecay}, with an extra decay~$t^{-\delta}$, but replacing~$\|v_1\|_{L^2}$ by~$s^{\delta-n\left(\frac1r-\frac12\right)}\,\|v_1\|_{L^r}$, for some~$\delta>0$, namely,
\begin{equation}\label{eq:r2qdecay}
\|v(t,\cdot)\|_{L^q} \leq C\,s^{\delta-n\left(\frac1r-\frac12\right)}\,t^{-\delta}\,\|v_1\|_{L^r} \,\times\begin{cases}
s\,t^{-n\left(\frac12-\frac1q\right)} & \text{if~$\mu\geq2$ or~$n(1-2/q)\leq\mu$,}\\
s^{1-n\left(\frac12-\frac1q\right)+\frac\mu2}\,t^{-\frac\mu2} & \text{if~$\mu\in(1,2)$ and~$n(1-2/q)\geq\mu$.}
\end{cases}\end{equation}
The use of~$s^{-n\left(\frac1r-\frac12\right)}\,\|v_1\|_{L^r}$ instead of~$\|v_1\|_{L^2}$ to estimate~$Fu$ is the key to obtain a global-in-time existence result in any space dimension~$n\geq3$. We also stress that the fact that~$\delta$ is positive is used only when~$p=p_\crit$. In the case~$p>p_\crit$, the proof would also work with~$\delta=0$.

We are now ready to estimate~$(Fu-Fw)(t,\cdot)$ in~$L^q$.

If~$\mu\geq2$ or~$n(1-2/q)\leq\mu$, by~\eqref{eq:r2qdecay} we obtain
\begin{align*}
& t^{\delta+n\left(\frac12-\frac1q\right)}\,\|(Fu-Fw)(t,\cdot)\|_{L^q}\\
    & \qquad \leq C\, \int_{t_0}^t s^{\delta+1-n\left(\frac1{r(q)}-\frac12\right)}\,\| (f(u)-f(w))(s,\cdot) \|_{L^r}\,ds \\
    & \qquad \leq C\, \int_{t_0}^t s^{\delta+1-n\left(\frac1{r(q)}-\frac12\right)-p\gamma_{r(q)p}}\,ds\,\|u-w\|_{X(T)}\big( \|u\|_{X(T)}^{p-1}+\|w\|_{X(T)}^{p-1}\big),
\end{align*}
where we used~\eqref{eq:fu} with H\"older inequality, and the fact that~$u,w\in X(T)$, to estimate (here we are using~$r(q)p\in[q_0,q_1]$, see Remark~\ref{rem:weak})
\begin{align}
\nonumber
\|(f(u)-f(w))(s,\cdot)\|_{L^{r(q)}}
    & \leq C\, \|(u-w)(s,\cdot)\|_{L^{r(q)p}}\,\big( \|u(s,\cdot)\|_{L^{r(q)p}}^{p-1}+\|w(s,\cdot)\|_{L^{r(q)p}}^{p-1}\big) \\
\label{eq:furq}
   & \leq C\,s^{-p\gamma_{r(q)p}}\,\|u-w\|_{X(T)}\big( \|u\|_{X(T)}^{p-1}+\|w\|_{X(T)}^{p-1}\big).
\end{align}
We now want to prove that
\[ \int_{t_0}^t s^{\delta+1-n\left(\frac1{r(q)}-\frac12\right)-p\gamma_{r(q)p}}\,\,ds \leq C\,t^\delta, \]
for any~$p\geq p_\crit$, that is,
\[ p_\crit\gamma_{r(q)p_\crit} + n\left(\frac1{r(q)}-\frac12\right) \geq 2. \]
Thanks to~\eqref{eq:rqp}, $r(q)p_\crit\leq q$, so that~$n(1-2/(r(q)p_\crit))\leq n(1-2/q)\leq\mu$, and we may replace
\[ p_\crit\,\gamma_{r(q)p_\crit} + n\left(\frac1{r(q)}-\frac12\right) = p_\crit \,n\left(\frac{1}2-\frac1{r(q)p_\crit}\right) +  n\left(\frac1{r(q)}-\frac12\right) = \frac{n(p_\crit-1)}2. \]
From this, we find~$p_\crit=1+4/n$.

On the other hand, if~$\mu\in(1,2)$ and~$n(1-2/q)\geq\mu$, by~\eqref{eq:r2qdecay} we obtain
\begin{align*}
& t^{\delta+\frac\mu2}\,\|(Fu-Fw)(t,\cdot)\|_{L^q}\\
    & \qquad \leq C\,\int_{t_0}^t s^{\delta+1-n\left(\frac1{r(q)}-\frac12\right)-n\left(\frac12-\frac1q\right)+\frac\mu2}\,\| (f(u)-f(w))(s,\cdot) \|_{L^r}\,ds \\
    & \qquad \leq C\,\int_{t_0}^t s^{\delta+1-n\left(\frac1{r(q)}-\frac1q\right)+\frac\mu2-p\gamma_{r(q)p}}\,ds\,\|u-w\|_{X(T)}\big( \|u\|_{X(T)}^{p-1}+\|w\|_{X(T)}^{p-1}\big).
\end{align*}
We now want to prove that
\[ \int_{t_0}^t s^{\delta+1-n\left(\frac1{r(q)}-\frac1q\right)+\frac\mu2-p\gamma_{r(q)p}}\,\,ds \leq C\,t^\delta, \]
for any~$p>p_\crit$, that is,
\[ p_\crit\gamma_{r(q)p_\crit} + n\left(\frac1{r(q)}-\frac1q\right)-\frac\mu2 \geq 2. \]
We distinguish two cases. If~$n(1-2/(r(q)p_\crit))\leq\mu$, as in the previous case, it is sufficient to estimate~$-\mu\geq -n(1/2-1/q)$, and proceed as before. On the other hand, if~$n(1-2/(r(q)p_\crit))\geq \mu$, using~\eqref{eq:rq}, we compute
\[ p_\crit\gamma_{r(q)p_\crit} + n\left(\frac1{r(q)}-\frac1q\right)-\frac\mu2 = (p_\crit-1)\frac\mu2 + \frac{n+1}2 -\frac{n-1}q.  \]
By using~$p_\crit-1=4/n$, $\mu\geq 2n/(n+3)$, and~$q\geq q_0$, so that~$-1/q\geq -1/q_0$, we may estimate
\[ (p_\crit-1)\frac\mu2 + \frac{n+1}2 -\frac{n-1}q \geq \frac4{n+3} +\frac1{q_0} \left(\frac{n+1}2\,q_0-(n-1)\right)= \frac4{n+3} +\frac4{q_0} = 2. \]
Summarizing, we proved that
\[ \|(Fu-Fw)(t,\cdot)\|_{L^q} \leq C(t_0)\,t^{-\gamma_q}\,\|u-w\|_{X(T)}\big( \|u\|_{X(T)}^{p-1}+\|w\|_{X(T)}^{p-1}\big), \]
for any~$q\in[q_0,q_1]$. In the case~$n=3$ and~$\mu=1$, we proceed as before, but we modify~\eqref{eq:normaX2} with
\begin{equation}\label{eq:normaX2mu1}
\|u\|_{X(T)} = \sup \big\{ t^{\frac12}\,(1+\log(t/t_0))^{-1}\,\|u(t,\cdot)\|_{L^3}, \ t^{\gamma_q}\,\|u(t,\cdot)\|_{L^q}: \ q\in(3,4], \quad t\in[t_0,T] \big\} \Bigr).
\end{equation}
This concludes the proof of Theorem~\ref{thm:L2}.
\end{proof}

%%%%%%%%%%%%%%%%%%%%%%%%%%%%%%%%%%%%%

%%%%%%%%%%%%%%%%%%%%%%%%%%%%%%%%%%%

\end{document}